\documentclass[a4paper, 11pt]{article}

\usepackage[arxiv]{optional}
\usepackage[english]{babel}
\usepackage[T1]{fontenc}
\usepackage{amsmath}
\opt{arxiv}{%
\usepackage{amsthm}
}%
\usepackage{amsfonts}
\usepackage{amssymb}
\usepackage[final]{graphicx}
\usepackage{paralist}
\usepackage{hyperref}
\usepackage{tikz}
\usepackage{booktabs}

\usepackage{cleveref}
\usepackage{subfigure}
\usepackage{pifont}
\usepackage{caption}
\usepackage{mathtools}

\opt{draft}{%
\usepackage{refcheck}

}

\opt{arxiv}{%
\theoremstyle{plain}
\newtheorem{theorem}{Theorem}[section]
\newtheorem{proposition}[theorem]{Proposition}
\newtheorem{lemma}[theorem]{Lemma}
\newtheorem{corollary}[theorem]{Corollary}
\newtheorem{conject}[theorem]{Conjecture}

\theoremstyle{definition}
\newtheorem{definition}[theorem]{Definition}
\newtheorem{example}[theorem]{Example}
\newtheorem{remark}[theorem]{Remark}
}%
\opt{springer}{%

}%

\newcommand{\eps}{\varepsilon}

\DeclareMathOperator{\im}{Im}

\DeclareMathOperator{\re}{Re}

\DeclareMathOperator{\sech}{sech}

\opt{springer}{%

}
\opt{arxiv}{%

}

\DeclarePairedDelimiter{\abs}{\lvert}{\rvert}
\DeclarePairedDelimiter{\norm}{\lVert}{\rVert}

\DeclarePairedDelimiter{\cc}{[}{]} 
\DeclarePairedDelimiter{\co}{[}{[} 

\newcommand{\R}{\ensuremath{\mathbb{R}}}
\newcommand{\C}{\ensuremath{\mathbb{C}}}

\newcommand{\N}{\ensuremath{\mathbb{N}}}

\newcommand{\cC}{\mathcal{C}}
\newcommand{\cO}{\mathcal{O}}

\newcommand{\conj}[1]{\overline{#1}}

\usepackage{algorithm}
\usepackage{algpseudocode}
\usepackage{algorithmicx}

\newenvironment{algo}{
\begin{algorithm}[t]}{
\end{algorithm}}

\usepackage{listings}
\usepackage[numbered]{matlab-prettifier}
\lstMakeShortInline"

\numberwithin{equation}{section}

\title{\textsc{A Newton method for harmonic mappings in the plane}}

\author{Olivier S\`{e}te\footnotemark[1] \and Jan Zur\footnotemark[1]}
\date{July 29, 2019}

\begin{document}
\maketitle
\renewcommand{\thefootnote}{\fnsymbol{footnote}}
\footnotetext[1]{TU Berlin, Department of Mathematics, MA 3-3, Stra{\ss}e des 
17. Juni 136, 10623 Berlin, Germany. \texttt{\{sete,zur\}@math.tu-berlin.de}}
\captionsetup[figure]{skip=0pt}
\begin{abstract}
We present an iterative root finding method for harmonic mappings in the 
complex plane, which is a generalization of Newton's method for analytic 
functions.
The complex formulation of the method allows an analysis in a complex variables 
spirit.
For zeros close to poles of $f = h + \conj{g}$ we construct initial 
points for which the harmonic Newton iteration is guaranteed to converge. 
Moreover, we study the number of solutions of $f(z) = \eta$ close to the 
critical set of $f$ for certain $\eta \in \C$.
We provide a Matlab implementation of the method, and
illustrate our results with several examples and numerical experiments, 
including phase plots and plots of the basins of attraction.

\end{abstract}
\paragraph*{Keywords:}
Zeros of harmonic mappings; Newton's method;
basins of attraction;
domain coloring;
Wilmshurst's conjecture; gravitational lensing.

\paragraph*{Mathematics Subject Classification (2010):}
30C55; 
30D05; 
31A05; 
37F99; 
65E05. 

\section{Introduction}

We study harmonic mappings, i.e., functions with local decomposition
\begin{equation}\label{eqn:harmonic_function_intro}
f(z) = h(z) + \conj{g(z)}
\end{equation}
in the complex plane, where $h$ and $g$ are analytic.  The function $f$ itself 
is not analytic in general, as it consists of an analytic and an anti-analytic 
term.

The modern treatment of harmonic mappings started from the landmark paper 
\cite{ClunieSheilSmall1984} of Clunie and Sheil-Small about univalent harmonic 
mappings in the plane. See also the comprehensive textbook of Duren~\cite{Duren2004}. 
While~\cite{Duren2004} considers univalent harmonic mappings, also the multivalent 
case has been intensively studied; see e.g. the collection of open problems 
\cite{BshoutyLyzzaik2010} by Bshouty and Lyzzaik. Problem~3.7 in \cite{BshoutyLyzzaik2010} 
deals with the maximum number of zeros of harmonic polynomials
$f(z) = p(z) + \conj{q(z)}$. 
Wilmshurst proved an upper bound for this number and 
conjectured another bound depending on $\deg(p)$ and $\deg(q)$~\cite{Wilmshurst1998}; 
see Section~\ref{sect:harmonic_polynomials} and the more recent 
publications concerning Wilmshurst's conjecture 
\cite{LeeLerarioLundberg2015,HauensteinLerarioLundbergMehta2015,KhavinsonLeeSaez2018}. 
Khavinson and {\'S}wi{\c{a}}tek~\cite{KhavinsonSwiatek2003}, and Geyer 
\cite{Geyer2008} 
settled the conjecture, including sharpness, for the special case $f(z) = p(z) 
- \conj{z}$.
Khavinson and Neumann \cite{KhavinsonNeumann2006} generalized these results 
to rational harmonic functions $f(z) = r(z) - \conj{z}$; see 
also~\cite{LuceSeteLiesen2014a,LiesenZur2018b}. 
The sharpness of the bound in the rational case was previously known due to the astrophysicist
Rhie~\cite{Rhie2003}. As it turns out, harmonic mappings model
gravitational lensing - an astrophysical phenomenon, where, due to deflection by massive 
objects, a light source seems brightened, distorted or multiplied for an observer. We refer 
to the expository articles \cite{KhavinsonNeumann2008,Petters2010}, and the survey 
\cite{BeneteauHudson2018}. More recent articles on harmonic mappings with applications to 
gravitational lensing are~\cite{BergweilerEremenko2010,KhavinsonLundberg2010,LuceSeteLiesen2014b,SeteLuceLiesen2015b,SeteLuceLiesen2015a,LuceSete2017,LiesenZur2018a}.

The results about the zeros of harmonic mappings published so far are more 
theoretical, and we are not aware of any specialized method to compute them, 
which may explain the lack of (numerical) examples in 
the respective publications. In this paper we focus on the numerical computation 
of zeros of harmonic mappings with a Newton method.

The global behavior of Newton's method on $\R^2$, see 
e.g.~\cite[Ch.~7]{PeitgenRichter1986}, has been less treated than the one of 
Newton's method for analytic or anti-analytic functions in $\C$.
However, the recent publications~\cite{DeLeo2018,DeLeo2019} give new impulses 
to this subject.
Since harmonic mappings are generalizations of both analytic and anti-analytic 
functions, our work may also lead to new insights on the dynamics of Newton's 
method on $\R^2$.
Our approach is not covered by the standard theory of complex 
dynamics~\cite{CarlesonGamelin1993,Milnor2006}, nor by anti-holomorphic 
dynamics, see~\cite{NakaneSchleicher2003,MukherjeeNakaneSchleicher2017}, since 
the corresponding Newton map~\eqref{eqn:newton_map} 
is neither analytic nor harmonic.

The paper is organized as follows. In Section~\ref{sect:prelim} we recall the definition and properties of 
harmonic mappings as well as Newton's method in Banach spaces. We then derive 
the harmonic Newton iteration and more generally a Newton iteration in complex 
notation for
non-analytic but real differentiable functions in Section~\ref{sect:iter}.
We investigate zeros of $f = h + \conj{g}$ close to poles 
in Section~\ref{sect:convergence}.
In particular, we construct initial points from the coefficients of the Laurent 
series of $h$ and $g$, for which the harmonic Newton iteration is guaranteed to 
converge to zeros of $f$.  In Section~\ref{sect:critical}, we study solutions 
of $f(z) = \eta$, where $z$ is close 
to a singular zero $z_0$ of $f$, and for certain small $\eta \in \C$.
In Section~\ref{sect:examples}, we consider several (numerical) examples.
Finally, we give a summary and discuss possible future research in 
Section~\ref{sect:outlook}.

\section{Mathematical background}\label{sect:prelim}

In this section we recall properties of harmonic mappings and classical 
convergence results for Newton's method.

\subsection{Harmonic mappings}

The \emph{Wirtinger derivatives} of a complex function $f$ are
\begin{equation*}
\partial_z f = \frac{1}{2} ( \partial_x f - i \partial_y f ),
\quad
\partial_{\conj{z}} f = \frac{1}{2} ( \partial_x f + i \partial_y f ),
\end{equation*}
where $z = x + iy$ with $x, y \in \R$; see~\cite[Section~1.2]{Duren2004},
\cite[Section~1.4]{Remmert1991}, or~\cite[p.~144]{Wegert2012}.
They satisfy
\begin{equation} \label{eqn:wirtinger_rules}
\partial_x f = \partial_z f + \partial_{\conj{z}} f, \quad
\partial_y f = i (\partial_z f - \partial_{\conj{z}} f), \quad
\partial_{\conj{z}} \conj{f(z)} = \conj{\partial_z f(z)}.
\end{equation}

A \emph{harmonic mapping} is a function $f : \Omega \to \C$ defined on an open 
set $\Omega \subseteq \C$ and with
\begin{equation*}
\Delta f = \partial_{xx} f + \partial_{yy} f
= 4 \partial_{\conj{z}} \partial_z f = 0.
\end{equation*}
Such a function has a local decomposition $f(z) = h(z) + \conj{g(z)}$,
where $h$ and $g$ are analytic functions of $z$ which are unique
up to an additive constant~\cite[p.~412]{DurenHengartnerLaugesen1996} 
or~\cite[p.~7]{Duren2004}.  The 
\emph{Jacobian} of a harmonic mapping $f$ at $z \in \Omega$ is
\begin{equation} \label{eqn:jacobian}
J_f(z) = \abs{\partial_z f(z)}^2 - \abs{\partial_{\conj{z}} f(z)}^2
= \abs{h'(z)}^2 - \abs{g'(z)}^2.
\end{equation}
We call $f$ \emph{sense-preserving} (or orientation-preserving) at $z \in 
\Omega$ if $J_f(z) > 0$, 
\emph{sense-reversing} if $J_f(z) < 0$, and \emph{singular} if $J_f(z) = 0$,
and similarly for zeros of $f$, e.g., $z_0$ is a singular zero of $f$ if 
$J_f(z_0) = 0$.
The points where $f$ is singular form the \emph{critical set}
\begin{equation} \label{eqn:crit}
\cC = \{ z \in \C : J_f(z) = 0 \}.
\end{equation}

One way to visualize complex functions are \emph{phase 
plots}, where the domain is colored according to the phase 
$f(z)/\abs{f(z)}$ of $f$.  We use the standard color scheme described 
in~\cite{Wegert2012} for all phase plots, and accordingly use white for 
the value $\infty$ (e.g., for poles of $h$ or $g$) and black for zeros; 
see Figure~\ref{fig:pp_intro}.
A comprehensive discussion of phase plots can be found in~\cite{Wegert2012}.
\begin{figure}
{\centering
\includegraphics[width=0.98\linewidth]{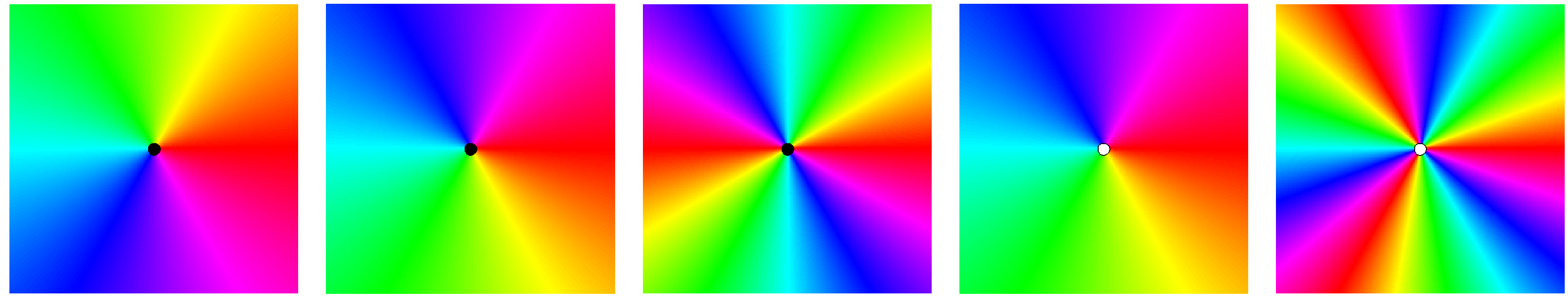}

}
\caption{Phase plots of $z$, $\conj{z}$, $z^2$, $1/z$ and $1/\conj{z}^3$ (from 
left to right).}
\label{fig:pp_intro}
\end{figure}

\begin{example} \label{ex:mpw}
Consider the rational harmonic function
\begin{equation} \label{eqn:mpw}
f(z) = h(z) + \conj{g(z)} = \frac{z^{n-1}}{z^n - r^n} - \conj{z},
\end{equation}
where $h$ has the poles $r e^{i 2 \pi k/n}$, $k = 0, 1, \ldots, n-1$.
The function $f$ has $3n+1$ zeros provided that $r > 0$ is sufficiently small; 
see~\cite{LuceSeteLiesen2014a} for a detailed analysis.
Figure~\ref{fig:mpw} shows the phase plot of $f$ for $n = 3$ and $r = 0.6$, 
the $3$ poles of $f$ and the $10$ zeros.
Depending on the orientation, the phase plot of $f$ near a zero looks 
like the phase plot of $z$ (sense-preserving) or $\conj{z}$ (sense-reversing) 
near $0$ in Figure~\ref{fig:pp_intro}.

\begin{figure}
{\centering
\includegraphics[width=0.48\linewidth]{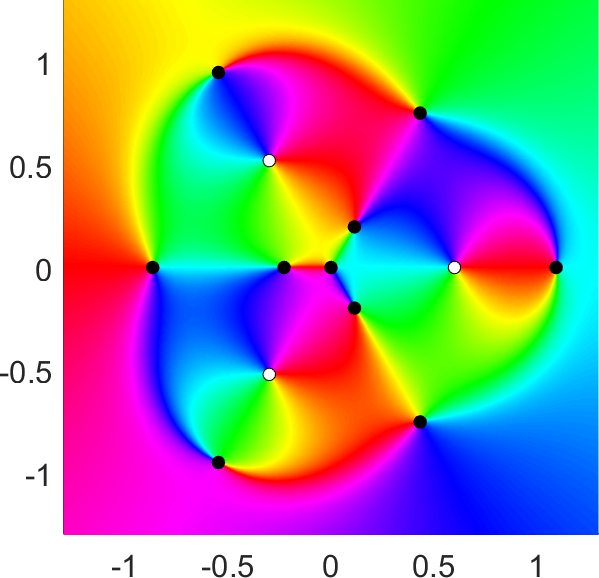}
\includegraphics[width=0.48\linewidth]{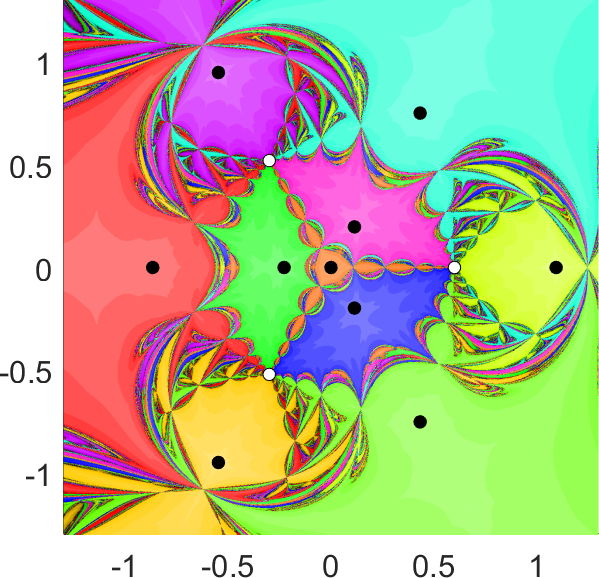}
\vspace{2mm}

}
\caption{Left: Phase plot of $f(z) = z^2 / (z^3 - 0.6^3) - \conj{z}$; see 
Example~\ref{ex:mpw}. White dots are the poles, black dots are the zeros.
Right: Basins of attraction of the zeros of $f$ in the harmonic Newton method; 
see Section~\ref{sect:iter} below.}
\label{fig:mpw}
\end{figure}
\end{example}

\subsection{Newton's method in Banach spaces}
\label{sect:classical_newton}

Let $F : D \to Y$, $D \subseteq X$ open, be a continuously Fr\'echet 
differentiable map
between two Banach spaces $X$, $Y$.  The Newton iteration with initial point 
$x_0 \in D$ is
\begin{equation} \label{eqn:classical_newton}
x_{k+1} = x_k - F'(x_k)^{-1} F(x_k), \quad k \geq 0.
\end{equation}
Under some regularity conditions,
the sequence of Newton iterates $(x_k)_k$ is quadratically convergent
($\norm{x_{k+1} - x_*} \leq \text{const} \cdot \norm{x_k - x_*}^2$), if $x_0$ 
is sufficiently close to a zero of $F$; see e.g.~\cite[Prop.~5.1]{Zeidler1986}.
The next two theorems quantify ``sufficiently close''.

The Newton-Kantorovich theorem guarantees existence of a zero close to the 
initial point of the Newton iteration.
In the following, $B(x_0; r)$ and $\overline{B}(x_0; r)$ denote the open 
and closed balls with center $x_0$ and radius $r$.

\begin{theorem}[Newton-Kantorovich, {\cite[Thm.~2.1]{Deuflhard2011}}] 
\label{thm:kantorovich}
Let $F : D \to Y$ be a continuously Fr\'echet differentiable map with $D 
\subseteq X$ open and convex.  For an initial point $x_0 \in D$ let $F'(x_0)$ 
be invertible.  Suppose that
\begin{align*}
\norm{F'(x_0)^{-1} F(x_0)} &\leq \alpha, \\
\norm{F'(x_0)^{-1} (F'(y)-F'(x))} &\leq \omega_0 \norm{y-x} \quad \text{for all 
} x, y \in D.
\end{align*}
Let $h_0 = \alpha \omega_0$ and $\rho = (1 - \sqrt{1 - 2 h_0})/\omega_0$, and 
suppose that
\begin{align*}
h_0 \leq \frac{1}{2} \quad \text{and} \quad
\overline{B}(x_0; \rho) \subseteq D.
\end{align*}
Then the sequence $(x_k)_k$ of Newton iterates is well defined, remains in 
\linebreak
$\overline{B}(x_0; \rho)$, and converges to some $x_*$ with $F(x_*) = 0$.
If $h_0 < \frac{1}{2}$, then the convergence is quadratic.
\end{theorem}

The refined Newton-Mysovskii theorem guarantees uniqueness of a zero, but not 
its 
existence.  Given a zero, it quantifies a neighborhood in which the Newton 
iteration will converge, as well as a constant for the quadratic convergence 
estimate.

\begin{theorem}[refined Newton-Mysovskii theorem, 
{\cite[Thm.~2.3]{Deuflhard2011}}]
\label{thm:mysovskii}
Let $F : D \to \R^n$ be a continuously differentiable map with $D \subseteq 
\R^n$ open and convex.  Suppose that $F'(x)$ is invertible for each $x \in D$.  
Suppose that
\begin{equation} \label{eqn:omega_mysovskii}
\norm{ F'(x)^{-1} (F'(y) - F'(x)) (y-x)} \leq \omega \norm{y-x}^2 \quad 
\text{for all } x, y \in D.
\end{equation}
Let $F(x) = 0$ have a solution $x_* \in D$.
For the initial point $x_0$ suppose that $\conj{B}(x_*; \norm{x_0 - x_*}) 
\subseteq D$ and that
\begin{equation*}
\norm{x_0 - x_*} < \frac{2}{\omega}.
\end{equation*}
Then the sequence of Newton iterates $(x_k)_k$ is well defined,
remains in the open ball $B(x_*; \norm{x_0 - x_*})$, converges to $x_*$, and
fulfills
\begin{equation*}
\norm{x_{k+1} - x_*} \leq \frac{\omega}{2} \norm{x_k - x_*}^2.
\end{equation*}
Moreover, the solution $x_*$ is unique in $B(x_*; 2/\omega)$.
\end{theorem}

We frequently use Landau's $\cO$-notation for complex functions.
By $f(z) + \cO(z^k)$ we mean an 
expression $f(z) + \psi(z)$, where $\psi(z) \in \cO(z^k)$, i.e., 
$\abs{\psi(z)/z^k}$ is 
bounded from above by a constant (usually for $z \to 0$ or $z \to \infty$, 
depending on the context).

\begin{remark}
Using the Taylor expansion $\sqrt{1 + z} = 1 + \frac{z}{2} + \cO(z^2)$, we find 
the  asymptotic expression
\begin{equation} \label{eqn:rho_minus_asympt}
\rho = \frac{1 - \sqrt{1 - 2 h_0}}{\omega_0}
= \frac{1 - (1 - h_0 + \cO(h_0^2))}{\omega_0}
= \alpha + \cO(\alpha h_0),
\end{equation}
in Theorem~\ref{thm:kantorovich},
whenever the \emph{Kantorovich quantity} $h_0$ is sufficiently small.
\end{remark}

While we only require the above results, other convergence results for Newton's 
method in Banach spaces exist; see e.g.~\cite{Smale1986,Wang1999}.

\section{The harmonic Newton method}\label{sect:iter}

We derive the harmonic Newton iteration and discuss its implementation.

\subsection{The Newton iteration in the plane in complex notation}
\label{subsect:iter}

Let $f : \Omega \to \C$ be (continuously) differentiable with respect to $x = 
\re(z)$ and $y = 
\im(z)$, where we do not require for the moment that $f$ is analytic or 
harmonic.
We identify $\C$ with $\R^2$, and $f$ with
\begin{equation} \label{eqn:F}
F : \Omega \to \R^2, \quad
F(z) = \begin{bmatrix} \re(f(z)) \\ \im(f(z)) \end{bmatrix}, \quad z = x + iy,
\end{equation}
which is a smooth function of $x$ and $y$.

Since $\re$ and $\im$ are 
$\R$-linear and continuous, they commute with $\partial_x$ and~$\partial_y$.  
Thus, $F'(z)$ has the matrix representation
\begin{equation}
F'(z) 
= \begin{bmatrix}
\re(\partial_z f(z)) + \re(\partial_{\conj{z}} f(z)) &
- \im(\partial_z f(z)) + \im(\partial_{\conj{z}} f(z)) \\
\im(\partial_z f(z)) + \im(\partial_{\conj{z}} f(z)) &
\re(\partial_z f(z)) - \re(\partial_{\conj{z}} f(z))
\end{bmatrix}, \label{eqn:dF}
\end{equation}
where we used the Wirtinger derivatives~\eqref{eqn:wirtinger_rules}.
The inverse of $F'(z)$ is
\begin{equation*}
F'(z)^{-1}
= \frac{1}{J_f(z)}
\begin{bmatrix}
\re(\partial_z f(z)) - \re(\partial_{\conj{z}} f(z)) &
\im(\partial_z f(z)) - \im(\partial_{\conj{z}} f(z)) \\
-\im(\partial_z f(z)) - \im(\partial_{\conj{z}} f(z)) &
\re(\partial_z f(z)) + \re(\partial_{\conj{z}} f(z))
\end{bmatrix},
\end{equation*}
provided $J_f(z) = \abs{\partial_z f(z)}^2 - \abs{\partial_{\conj{z}} f(z)}^2 
\neq 0$; 
compare~\eqref{eqn:jacobian}.  We identify $F'(z)$ and its inverse with the
$\R$-linear maps on $\C$
\begin{align*}
F'(z)(w) &= (\partial_z f(z)) w + (\partial_{\conj{z}} f(z)) \conj{w}, \\
F'(z)^{-1}(w) &= \frac{1}{J_f(z)} ( \conj{\partial_z f(z)} w - 
(\partial_{\conj{z}} f(z)) \conj{w}).
\end{align*}
Then the Newton iteration~\eqref{eqn:classical_newton} for $F$ can be rewritten 
in $\C$ as
\begin{equation} \label{eqn:complex_newton}
z_{k+1} = z_k - \frac{\conj{\partial_z f(z_k)} f(z_k) - \partial_{\conj{z}} 
f(z_k) \conj{f(z_k)}}{\abs{\partial_z f(z_k)}^2 - \abs{\partial_{\conj{z}} 
f(z_k)}^2}, \quad k \geq 0.
\end{equation}

\subsection{The harmonic Newton iteration}

For a harmonic mapping with local decomposition $f = h + \conj{g}$, we have
$\partial_z f(z) = h'(z)$ and $\partial_{\conj{z}} f(z) = 
\partial_{\conj{z}} \conj{g(z)} = \conj{g'(z)}$
with~\eqref{eqn:wirtinger_rules},
so that $F'(z)$ and $F'(z)^{-1}$ can be identified with
\begin{align}
F'(z)(w) &= h'(z) w + \conj{g'(z) w}, \label{eqn:Fp} \\
F'(z)^{-1}(w) &= \frac{1}{J_f(z)} ( \conj{h'(z)} w - \conj{g'(z) w}),
\label{eqn:Fpinv}
\end{align}
and the Newton iteration~\eqref{eqn:complex_newton} takes the form
\begin{equation} \label{eqn:harmonic_newton}
z_{k+1} = z_k - \frac{\conj{h'(z_k)} f(z_k) - \conj{g'(z_k) 
f(z_k)}}{\abs{h'(z_k)}^2 - \abs{g'(z_k)}^2}, \quad k \geq 0,
\end{equation}
which we call the \emph{harmonic Newton iteration}.
It reduces to the classical Newton iteration $z_{k+1} = z_k - f(z_k)/f'(z_k)$ 
when $f$ is analytic ($f = h$).  Similarly, if $f$ is anti-analytic
($f = \conj{g}$), the iteration simplifies to $z_{k+1} = z_k - g(z_k)/g'(z_k)$.
Hence, the harmonic Newton iteration~\eqref{eqn:harmonic_newton} is a 
natural generalization of the classical Newton iteration.

We define the \emph{harmonic Newton map}
\begin{equation} \label{eqn:newton_map}
H_f : \Omega \setminus \cC \to \C, \quad
z \mapsto H_f(z) = z - \frac{\conj{h'(z)} f(z) - \conj{g'(z) 
f(z)}}{\abs{h'(z)}^2 - \abs{g'(z)}^2},
\end{equation}
which is in general neither analytic nor harmonic.  Note that $z_* \in \Omega 
\setminus \cC$ is a fixed point of $H_f$ if and only if $z_*$ is a 
non-singular zero of $f$.
Borrowing notation from complex 
dynamics, see e.g.~\cite{CarlesonGamelin1993,Milnor2006}, we define the 
\emph{basin of attraction} of a zero $z_*$ of $f$ as
\begin{equation*}
A(z_*) = \{z \in \C : \lim_{k \to \infty} H_f^k(z) = z_* \},
\end{equation*}
where $H_f^1 = H_f$ and $H_f^k = H_f \circ H_f^{k-1}$ for $k \geq 2$.
Note that the basin of attraction of a non-singular zero $z_*$ contains an open 
neighborhood of $z_*$, by the local convergence of Newton's method.
For singular zeros this need not be the case; see Example~\ref{ex:einstein}.

To visualize the dynamics of the harmonic Newton map, we use the following 
standard domain coloring technique; see, e.g.,~\cite{Gilbert2001,Varona2002}.
We color every point according to which basin it belongs to.
The color level indicates the number of iterations, the darker the more 
iterations were required.
Points where the harmonic Newton iteration does not converge 
are colored in black; see Figure~\ref{fig:mpw}.

\subsection{Implementation}

\begin{algo}
\caption{\textsc{Harmonic Newton method}}\label{alg:HNM}
\textbf{Input:} $f$, $h'$, $g'$, and $z_0 \in \C$, $N \in \N$, $\eps_f, \eps_z 
> 0$ \\
\textbf{Output:} approximated zero of $f$
\begin{algorithmic}[1]
\For{$k = 0, \dots, N-1$}
\State $\displaystyle z_{k+1} = z_k - \frac{\conj{h'(z_k)} f(z_k) - 
\conj{g'(z_k)
f(z_k)}}{\abs{h'(z_k)}^2 - \abs{g'(z_k)}^2}$
\If {$\abs{f(z_{k+1})} < \eps_f$ \textbf{or} $\abs{z_{k+1} - z_k} < \eps_z 
\abs{z_{k+1}}$}
\State \textbf{return} $z_{k+1}$
\EndIf
\EndFor
\State \textbf{return} $z_{N}$
\end{algorithmic}
\end{algo}

Following Higham~\cite[Sect.~25.5]{Higham2002}, we stop the harmonic Newton 
iteration when the residual $\abs{f(z_k)}$ 
or the relative difference between two iterates $\abs{z_{k+1} - 
z_k}/\abs{z_{k+1}}$ are less than given tolerances, 
or when a prescribed maximum number of steps has been performed.
This gives the \emph{harmonic Newton method}; see Algorithm~\ref{alg:HNM}.
Note that, when $f(z_*) = 0$ and $\abs{z_k - z_*}$ is sufficiently small,
we have $\abs{z_{k+1} - z_*} \leq 2 c \abs{z_{k+1} - z_k}^2$, where $c$ is the 
constant from the (local) quadratic convergence of the Newton method; 
see~\cite[Sect.~25.5]{Higham2002}.

Our MATLAB implementation of Algorithm~\ref{alg:HNM} is displayed in 
Figure~\ref{fig:code}.
To apply a harmonic Newton step not only to a single point, but to a vector 
or matrix of points, we vectorize the iteration~\eqref{eqn:harmonic_newton}, 
see lines 25--26.  This also requires vectorized function handles for 
$f$, $h'$, $g'$. 
Vectorization is particularly useful when we have several initial points,
since applying the method to a matrix of 
initial points is usually faster than applying it to each point individually.
When all points are iterated 
simultaneously, some might already have converged while others have 
not.  However, further iterating a point that has numerically converged can be 
harmful in finite precision, e.g., when a point is close to the critical set, a 
further
Newton step may lead to division by zero.
Hence, we stop the iteration individually for each point, 
by removing them from the list of active points; see lines 15 and 33.
\begin{figure}[t!]
\footnotesize
\begin{lstlisting}
function [z, numiter] = harmonicNewton(f, dh, dg, z, maxit, restol, steptol)

% Set default values:
if ((nargin < 5) || isempty(maxit)), maxit = 50; end
if ((nargin < 6) || isempty(restol)), restol = 1e-14; end
if ((nargin < 7) || isempty(steptol)), steptol = 1e-14; end
% Setup
active = 1:numel(z);                    % track which points to iterate
numiter = (maxit+1)*ones(size(z));      % number of iterations
% Check initial points:
fz = f(z);
converged = (abs(fz) < restol);         % already converged
diverged = (isnan(fz) | isinf(fz));     % NaN or Inf
numiter(converged) = 0;                 % no Newton steps required
active(converged | diverged) = [];      % do not further iterate
fz = fz(active);
% Harmonic Newton iteration:
for kk = 1:maxit
    if (isempty(active)), return, end   % every point has converged
    zold = z(active);
    % Evaluate functions:
    dhz = dh(zold);
    dgz = dg(zold);
    % Newton step:
    z(active) = zold ...
        - (conj(dhz).*fz - conj(dgz.*fz))./(abs(dhz).^2 - abs(dgz).^2);  
    % Convergence check:
    fz = f(z(active));
    converged = (abs(fz) < restol) | ...
        (abs(z(active) - zold) < steptol*abs(zold));
    diverged = (isnan(fz) | isinf(fz)); % NaN or inf
    numiter(active(converged)) = kk;    % took kk iterations
    active(converged | diverged) = [];  % remove converged points
    fz(converged | diverged) = [];      % remove converged points
end
end
\end{lstlisting}
\normalsize
\caption{MATLAB implementation of the harmonic Newton method used in the 
examples.}
\label{fig:code}
\end{figure}
Note that computing the next iterate in~\eqref{eqn:harmonic_newton} requires 
only the three function evaluations $f(z_k)$, $h'(z_k)$ and $g'(z_k)$.

Computing the next iterate by~\eqref{eqn:harmonic_newton} or by solving the 
$2 \times 2$ real linear algebraic system corresponding to
\begin{equation} \label{eqn:linsys}
F'(z_k) (z_{k+1} - z_k) = - F(z_k)
\end{equation}
with~\eqref{eqn:dF} are equivalent in exact arithmetic.
This is not the case in finite precision, where the Jacobian $J_f(z)$ can be 
numerically zero, although $z$ is not on the critical set.  To avoid division 
by zero, it is preferable to compute the next iterate by 
solving~\eqref{eqn:linsys}, instead of inverting $F'(z_k)$ explicitly.
However, we use~\eqref{eqn:harmonic_newton} in our numerical experiments 
(except in Example~\ref{ex:tan}), see line~2 in Algorithm~\ref{alg:HNM} and 
lines 25--26 in Figure~\ref{fig:code}, because~\eqref{eqn:harmonic_newton} 
can be easily vectorized.

\section{Finding zeros close to poles}
\label{sect:convergence}

The complex formulation~\eqref{eqn:harmonic_newton} makes the harmonic Newton 
iteration ame\-nable to analysis in a complex variables spirit.
We investigate zeros of harmonic mappings $f = h + \conj{g}$ close to poles 
$z_0$, i.e., points where $\lim_{z \to z_0} \abs{f(z)} = \infty$; 
see~\cite[Def.~2.1]{SuffridgeThompson2000}.
More precisely, we construct initial points from the Laurent coefficients of 
$h$ and $g$, for which the harmonic Newton 
iteration is guaranteed to converge to zeros of~$f$.
We start with an example.

\begin{example} \label{ex:tan}
We consider $f(z) = \tan(z) - \conj{z}$; see Figure~\ref{fig:tan}.
\begin{figure}[t!]
{\centering
\includegraphics[width=0.9\linewidth]{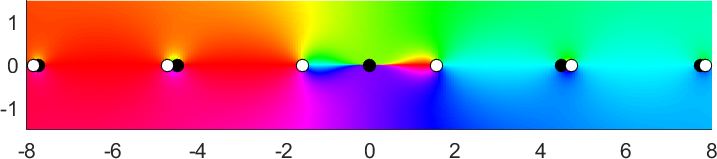}
\vspace{2mm}
\includegraphics[width=0.9\linewidth]{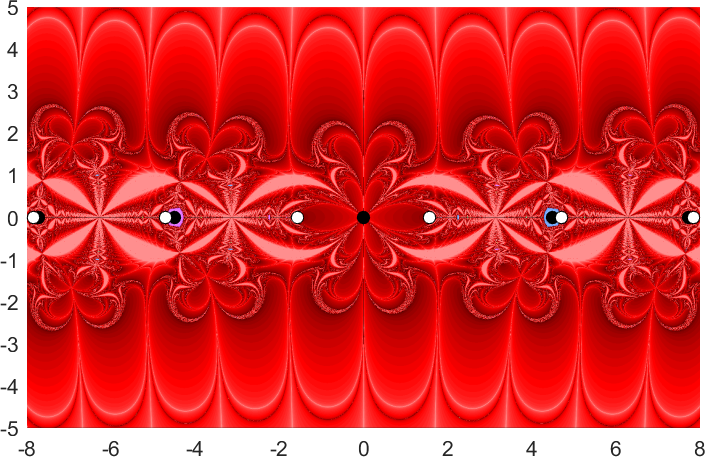}
\vspace{2mm}

}
\caption{The function $f(z) = \tan(z) - \conj{z}$ from Example~\ref{ex:tan}. 
Top: Phase plot, zeros (black dots), poles (white dots). Bottom: Basins of 
attraction.}
\label{fig:tan}
\end{figure}
The function has a singular zero at the origin, and one (real) zero close to 
each pole of the tangent, except $\pm \frac{\pi}{2}$.
To compute the zeros of $f$ we apply the harmonic Newton method 
to a grid of initial points in $\cc{-8, 8} \times \cc{-2, 2}$ with mesh size 
$0.2$, tolerances $10^{-14}$, and with up to $50$ steps.
The residual at our computed zeros satisfies $\abs{f(z_k)} \leq 1.5600 \cdot 
10^{-13}$.

Since the Jacobian  $J_f(z) = \abs{1 + \tan(z)^2}^2 - 1$ is numerically zero 
for $\abs{z} \leq 10^{-9}$, we use~\eqref{eqn:linsys} instead 
of~\eqref{eqn:harmonic_newton} to compute the basins of attraction in this 
example.
We observe that the basin of $0$ is huge, while the basins of zeros close to 
poles are small.  Moreover, the size of the basins decreases for larger poles, 
and the poles seem to be on the boundary of the basins.

Note that the black points on the imaginary axis are due to a numerically 
vanishing Jacobian, as above.  Analytically, the harmonic Newton map is
\begin{equation*}
H_f(iy) = i \left( y - \frac{y + \tanh(y)}{1 + \sech(y)^2} \right)
\quad \text{for } y \in \R,
\end{equation*}
and one can prove that every point on 
the imaginary axis converges to $0$.
\end{example}

Next, we construct initial points for which the harmonic Newton iteration 
converges to zeros close to a pole.  For this, we need the following lemma.

\begin{lemma} \label{lem:deg_one}
The equation $a z + \conj{b z} = c$ has a unique solution if and only if 
$\abs{a} \neq \abs{b}$.  In that case, the unique solution is
$z = \frac{\conj{a} c - \conj{b c}}{\abs{a}^2 - \abs{b}^2}$.
\end{lemma}

\begin{proof}
Splitting $a z + \conj{b z} = c$ in its real and imaginary part yields a real 
$2 \times 2$ linear algebraic system with determinant $\abs{a}^2 - \abs{b}^2$, 
from which we obtain the assertion.
\end{proof}

To apply the Newton-Kantorovich theorem to the harmonic Newton 
iteration~\eqref{eqn:harmonic_newton}, we identify $\C$ with $\R^2$ as in 
Section~\ref{subsect:iter}.
First, we derive bounds for the constants $\alpha$ and $\omega_0$ in 
Theorem~\ref{thm:kantorovich} using the formulas for $F'$ and 
$(F')^{-1}$ from~\eqref{eqn:Fp} and~\eqref{eqn:Fpinv}.
We find
\begin{equation*}
\abs{F'(z_0)^{-1} (w)}
\leq \frac{\abs{h'(z_0)} + \abs{g'(z_0)}}{\abs[\big]{\abs{h'(z_0)}^2 - 
\abs{g'(z_0)}^2}} \abs{w}
= \frac{\abs{w}}{\abs[\big]{\abs{h'(z_0)} - \abs{g'(z_0)}}},
\end{equation*}
and so
\begin{equation}
\abs{F'(z_0)^{-1} (F(z_0))} \leq \frac{\abs{f(z_0)}}{\abs[\big]{\abs{h'(z_0)} - 
\abs{g'(z_0)}}} \leq \alpha
\label{eqn:alpha}
\end{equation}
and
\begin{align}
\norm{F'(z_0)^{-1} (F'(y)-F'(x))}
&= \max_{\abs{z} = 1} \abs{F'(z_0)^{-1} (F'(y)-F'(x)) (z) } \nonumber \\
&\leq \frac{\max_{\abs{z} = 1} (\abs{h'(y) - h'(x)} + \abs{g'(y) - g'(x)}) 
\abs{z}}{\abs[\big]{\abs{h'(z_0)} - \abs{g'(z_0)}}} \nonumber \\
&\leq \frac{\sup_{t \in D} \abs{h''(t)} + \sup_{t \in D} 
\abs{g''(t)}}{\abs[\big]{\abs{h'(z_0)} - \abs{g'(z_0)}}} \abs{y-x} \nonumber \\
&\leq \omega_0 \abs{y-x}. \label{eqn:omega0}
\end{align}

Next we present the main result of this section.

\begin{theorem} \label{thm:zero_at_pole}
Let $f = h + \conj{g}$, where
\begin{equation*}
h(z) = \sum_{k=-n}^\infty a_k (z-z_0)^k, \quad
g(z) = \sum_{k=-n}^\infty b_k (z-z_0)^k,
\end{equation*}
with $n \geq 1$, and $\abs{a_{-n}} \neq \abs{b_{-n}}$.
Suppose that $c = - (a_0 + \conj{b}_0) \neq 0$, and
let $z_1, \ldots, z_n$ be the $n$ solutions of
\begin{equation} \label{eqn:init_at_pole}
(z - z_0)^n = \frac{\abs{a_{-n}}^2 - \abs{b_{-n}}^2}{\conj{a}_{-n} c - 
\conj{b}_{-n} \conj{c}}.
\end{equation}
We then have for sufficiently large $\abs{c}$:
\begin{enumerate}
\item There exist $n$ distinct zeros of $f$ near $z_0$.
\item The zeros in 1.\ are the limits of the harmonic Newton iteration with 
initial points $z_1, \ldots, z_n$.
\end{enumerate}
\end{theorem}

\begin{proof}
We apply the Newton-Kantorovich theorem for each of the $n$ initial points 
$z_j$.  Without loss of generality, we assume $z_0 = 0$.  Then ``$\abs{c}$ 
sufficiently large'' is equivalent to ``$\abs{z_j}$ sufficiently small''.

The cases $n = 1$ and $n \geq 2$ differ slightly, and we begin with $n \geq 2$.
First, we determine $\alpha$ from~\eqref{eqn:alpha}.  
With~\eqref{eqn:init_at_pole} and Lemma~\ref{lem:deg_one} we obtain
\begin{equation*}
a_{-n} (z_j-z_0)^{-n} + \conj{b}_{-n} (z_j-z_0)^{-n} = c = - (a_0 + \conj{b}_0),
\end{equation*}
and
\begin{equation} \label{eqn:bound_fzj}
\abs{f(z_j)} \leq (\abs{a_{-n+1}} + \abs{b_{-n+1}}) \abs{z_j}^{-n+1} + 
\cO(\abs{z_j}^{-n+2}),
\end{equation}
and from
\begin{equation*}
h'(z) = -n a_{-n} z^{-n-1} + \cO(z^{-n}), \quad
g'(z) = -n b_{-n} z^{-n-1} + \cO(z^{-n}),
\end{equation*}
also
\begin{equation*}
\abs[\big]{\abs{h'(z_j)} - \abs{g'(z_j)}}
= n \abs[\big]{\abs{a_{-n}} - \abs{b_{-n}}} \abs{z_j}^{-n-1} + 
\cO(\abs{z_j}^{-n}),
\end{equation*}
so that
\begin{equation*}
\begin{split}
\frac{\abs{f(z_j)}}{\abs[\big]{\abs{h'(z_j)} - \abs{g'(z_j)}}}
&\leq \frac{(\abs{a_{-n+1}} + \abs{b_{-n+1}}) \abs{z_j}^{-n+1} + 
\cO(\abs{z_j}^{-n+2})}{n \abs[\big]{\abs{a_{-n}} - \abs{b_{-n}}} 
\abs{z_j}^{-n-1} + \cO(\abs{z_j}^{-n})} \\
&= \frac{\abs{a_{-n+1}} + \abs{b_{-n+1}}}{n \abs[\big]{\abs{a_{-n}} - 
\abs{b_{-n}}}} \abs{z_j}^2 + \cO(\abs{z_j}^3) = \alpha.
\end{split}
\end{equation*}
Here we used that $1/(1+ \cO(x)) = 1 + \cO(x)$ for $x \to 0$.

Next we determine $\omega_0$ from~\eqref{eqn:omega0}.  Let $D = \{ z \in \C : 
\abs{z-z_j} < q \abs{z_j} \}$ for some $0 < q < 1$.  From
\begin{equation*}
h''(z) = n (n+1) a_{-n} z^{-n-2} + \cO(z^{-n-1}),
\end{equation*}
we get
\begin{equation*}
\sup_{t \in D} \abs{h''(t)} \leq n (n+1) \abs{a_{-n}} ((1-q) \abs{z_j})^{-n-2} 
+ \cO(\abs{z_j}^{-n-1}),
\end{equation*}
and similarly for $g$.  Then
\begin{equation*}
\begin{split}
&\frac{\sup_{t \in D} \abs{h''(t)} + \sup_{t \in D} 
\abs{g''(t)}}{\abs[\big]{\abs{h'(z_j)} - \abs{g'(z_j)}}} \\
&\leq
\frac{n (n+1) (\abs{a_{-n}} + \abs{b_{-n}}) ((1-q) \abs{z_j})^{-n-2} 
+ \cO(\abs{z_j}^{-n-1})}{n \abs[\big]{\abs{a_{-n}} - \abs{b_{-n}}} 
\abs{z_j}^{-n-1} + \cO(\abs{z_j}^{-n})} \\
&= \frac{(n+1) (\abs{a_{-n}} + \abs{b_{-n}}) 
(1-q)^{-n-2}}{\abs[\big]{\abs{a_{-n}} - \abs{b_{-n}}}} \abs{z_j}^{-1} + \cO(1) 
= \omega_0,
\end{split}
\end{equation*}
and $h_0 = \alpha \omega_0 = \cO(\abs{z_j})$.  Hence, for any $0 < q < 1$, 
we have $h_0 < \frac{1}{2}$ for sufficiently small $\abs{z_j}$.

It remains to show that $\rho < q \abs{z_j}$.
Using~\eqref{eqn:rho_minus_asympt} we get
\begin{equation}
\rho = \alpha + \cO(\alpha h_0)
= \cO(\abs{z_j}^2)
< q \abs{z_j}
\end{equation}
for sufficiently small $\abs{z_j}$.  By Theorem~\ref{thm:kantorovich},
the sequence of harmonic Newton iterates remains in the closed disk
$\overline{D}(z_j; \rho)$ and converges to a zero of $f$.

Finally, we show that the $n$ disks $\overline{D}(z_j; \rho)$, $j = 1, 2, 
\ldots, n$, are disjoint.  By construction, the $z_j$ are $n$-th 
roots.  Thus, they have equispaced angles, satisfy $\abs{z_j - z_{j+1}} = 2 
\sin(\pi/n) \abs{z_j}$, and have distance $\abs{z_j}$ from the origin.  
Since $\rho = \cO(\abs{z_j}^2)$, the disks are disjoint for sufficiently 
small $\abs{z_j}$.
This concludes the proof in the case $n \geq 2$.

For $n = 1$, the proof differs in~\eqref{eqn:bound_fzj}, where we have 
$\abs{f(z_1)} = \cO(\abs{z_1})$ (instead of $\cO(1)$).  Hence, $\alpha 
= \cO(\abs{z_1}^3)$ and the proof can be completed as for $n \geq 2$.
We omit the details.
\end{proof}

\begin{remark}
\begin{enumerate}
\item The assumption ``$\abs{c}$ is sufficiently large'' is necessary to 
guarantee the existence of a zero close to the pole;
see Figure~\ref{fig:tan}, where we have no zero ``close'' to the poles $\pm 
\frac{\pi}{2}$ and where $c = \pm \frac{\pi}{2}$, respectively.

\item One of the coefficients $a_{-n}$ and $b_{-n}$ in 
Theorem~\ref{thm:zero_at_pole} may be zero (but not both).  In particular, it 
is possible that only one of the functions $h$ and $g$ has a pole at $z_0$.

\item The case $\abs{a_{-n}} = \abs{b_{-n}}$ is excluded in 
Theorem~\ref{thm:zero_at_pole}.  Such functions can have nonisolated 
zeros, e.g., $f(z) = z^{-n} + \conj{z}^{-n} = 2 \re(z^{-n})$.

\item The Laurent coefficients of $h$ and $g$ are given by contour integrals, 
and can be computed numerically by quadrature, e.g. with the trapezoidal 
rule; see, e.g.,~\cite{TrefethenWeideman2014}.  To compute several 
Laurent coefficients at once, we can use the trapezoidal rule in combination 
with the DFT;
see~\cite[Sect.~3.8.1]{AblowitzFokas2003}, or~\cite[Sect.~13.4]{Henrici1986}.
\end{enumerate}
\end{remark}

By the transformation $z \mapsto \frac{1}{z}$, we obtain the following 
corollary for zeros ``close to infinity''.

\begin{corollary} \label{cor:zero_at_infty}
Let $f = h + \conj{g}$, where
\begin{equation*}
h(z) = \sum_{k = -\infty}^n a_k z^k, \quad
g(z) = \sum_{k = -\infty}^n b_k z^k, \quad \text{for } \abs{z} > R > 0,
\end{equation*}
with $n \geq 1$, and $\abs{a_n} \neq \abs{b_n}$.
Furthermore, let $c = -(a_0 + \conj{b}_0) \neq 0$ and 
let $z_1, \ldots, z_n$ be the $n$ solutions of
\begin{equation*}
z^n = \frac{\conj{a}_{n} c - \conj{b}_n \conj{c}}{\abs{a_n}^2 - \abs{b_n}^2}.
\end{equation*}
We then have for sufficiently large $\abs{c}$:
\begin{enumerate}
\item There exist $n$ distinct zeros of $f$ ``close to infinity''.
\item The zeros in 1.\ are the limits of the harmonic Newton iteration with 
initial points $z_1, \ldots, z_n$.
\end{enumerate}
\end{corollary}

\begin{example} \label{ex:poles}

\begin{figure}[t!]
{\centering
\includegraphics[width=0.48\linewidth]{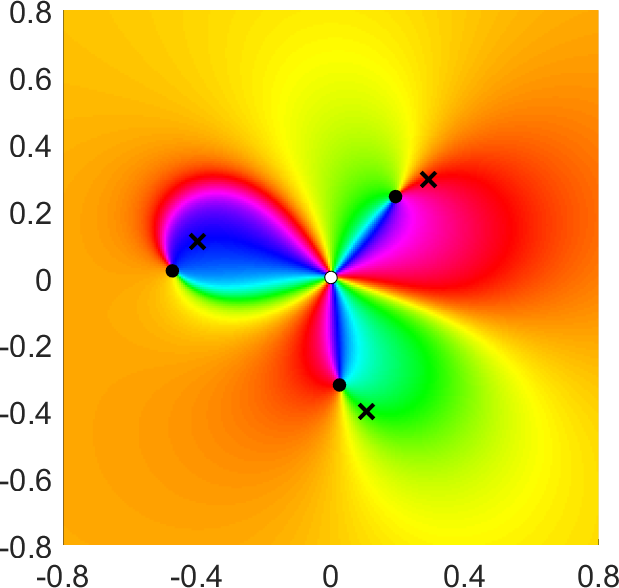}
\includegraphics[width=0.48\linewidth]{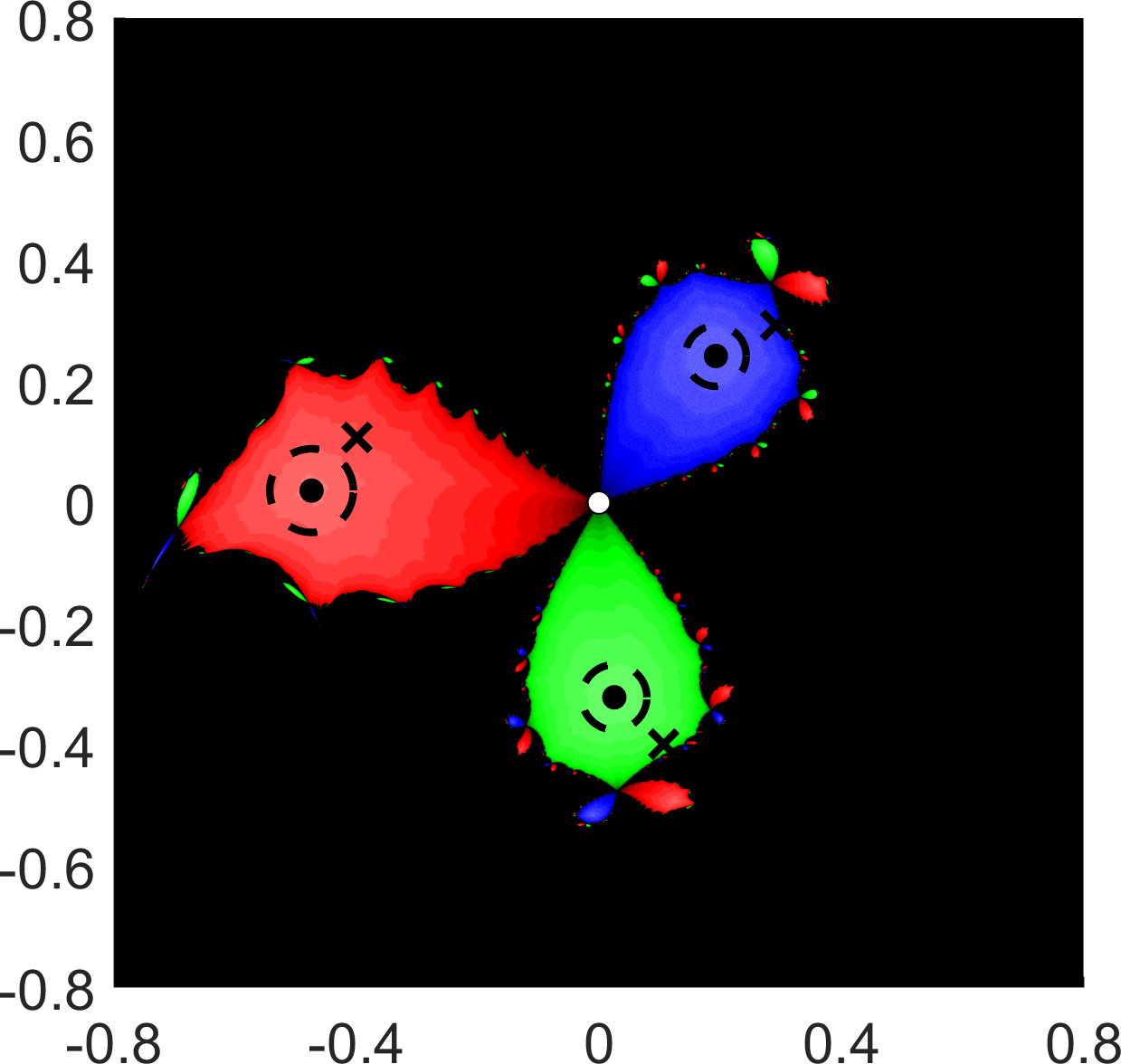}
\vspace{2mm}

\includegraphics[width=0.48\linewidth]{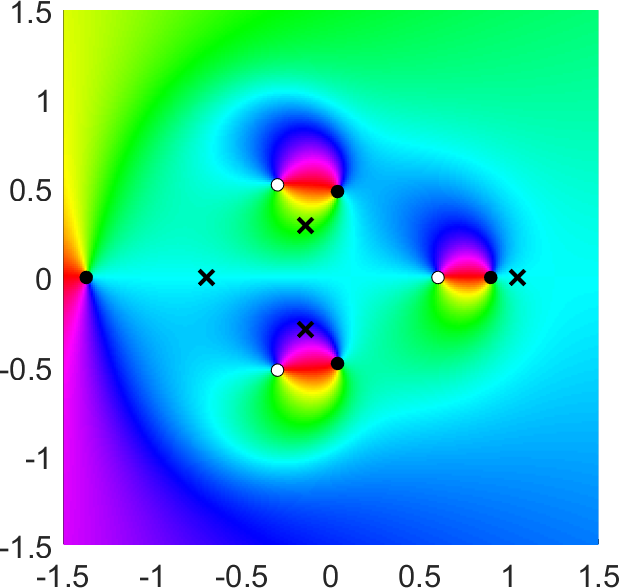}
\includegraphics[width=0.48\linewidth]{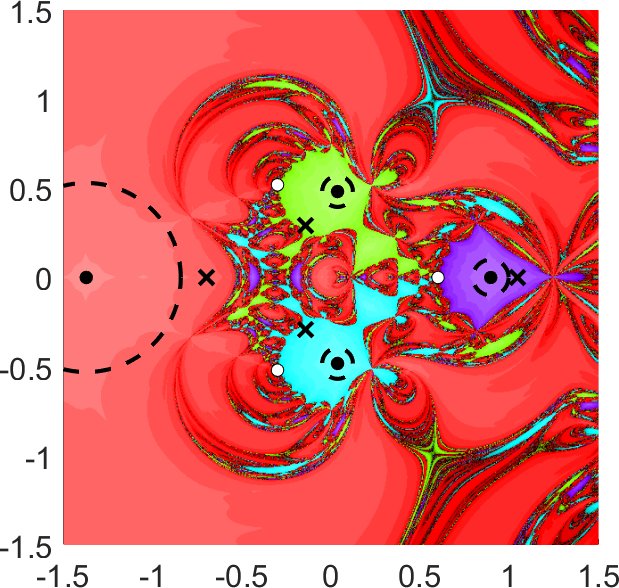}
\vspace{2mm}

}
\caption{Phase plots (left) and basins of attraction (right) of $f(z) = 
\frac{1}{z^3} + 10(1+i) + \conj{\left( \frac{2i}{z^2} \right)}$ (top) and
$f(z) = 
\frac{z^2}{z^3 - 0.6^3} - 0.7 - \conj{z}$ (bottom) from Example~\ref{ex:poles}.
Crosses indicate the initial points for the harmonic Newton method from 
Theorem~\ref{thm:zero_at_pole} and Corollary~\ref{cor:zero_at_infty}.
The dashed circles are regions of convergence guaranteed by 
Theorem~\ref{thm:mysovskii}.
}
\label{fig:poles}
\end{figure}
Let $f(z) = \frac{1}{z^3} + 10(1+i) + \conj{\left( \frac{2i}{z^2} \right)}$, 
where the 
analytic and anti-analytic part have poles of order $3$ and $2$ in $z_0 = 0$, 
with $c = - 10(1+i)$.
As predicted by Theorem~\ref{thm:zero_at_pole}, there are three distinct zeros 
close to $z_0 = 0$; see Figure~\ref{fig:poles}.  The initial points from the 
theorem (marked by crosses) lie close to the edges of the basins.
After $9$ steps of the harmonic Newton iteration, the maximal residual is 
$3.5527 \cdot 10^{-15}$.
For points in the black region, the sequence of harmonic Newton iterates tends
to $\infty$.
Note that the boundary of the basins of attraction seems to have fractal 
character; compare to the ``Douady rabbit''.

Next, we consider the function $f(z) = \frac{z^2}{z^3 - 0.6^3} - 0.7 - 
\conj{z}$, which is the function from Example~\ref{ex:mpw} with an additive 
perturbation.  The function has three simple poles, one zero close to each 
pole, and a fourth zero; see Figure~\ref{fig:poles}.  The crosses indicate the 
initial points from Theorem~\ref{thm:zero_at_pole} and from 
Corollary~\ref{cor:zero_at_infty} (for the fourth zero).
As for the previous example, they are close to the boundary of the basins of 
attraction of the zeros, which again suggests that the constant $c$ cannot be 
much smaller.
After $7$ steps of the harmonic Newton iteration, the maximal residual 
of the iterates is $4.3299 \cdot 10^{-15}$.

The plot of the basins of attraction of the two functions also show dashed 
circles, which are regions of convergence guaranteed by the refined 
Newton-Mysovskii 
theorem (Theorem~\ref{thm:mysovskii}).  For a zero $z_*$ of $f$, this region is 
computed as follows.
Since the theorem guarantees convergence for initial points in a disk, we
choose $D$ as a disk with radius $r$ and center $z_*$.  To estimate the 
constant 
$\omega$, we evaluate~\eqref{eqn:omega_mysovskii} on a grid in $D$.
The points in the disk with radius $\min \{ r, 2/\omega \}$ and center $z_*$
are guaranteed to converge to $z_*$, and we maximize over $r$ to compute the 
largest such disk.
We observe that our initial points are not covered by 
Theorem~\ref{thm:mysovskii}, but the iterates converge nevertheless to zeros of 
the functions.
\end{example}

\section{Behavior near the critical set}
\label{sect:critical}

The harmonic Newton map is not defined on the critical set $\cC$;
see Section~\ref{sect:iter}.  Nevertheless, in Examples~\ref{ex:tan} 
and~\ref{ex:einstein} below, the harmonic Newton method approximates singular 
zeros, i.e., $f(z_0) = 0$ with $z_0 \in \cC$,
reasonably well.
We now consider solutions of $f(z) = \eta$ with small $\abs{\eta}$ and $z$ 
close to the singular zero $z_0$.
For rational harmonic functions $f(z) = r(z) - \conj{z}$, these solutions 
play a crucial role in the study of the (global) number of zeros of $f$; 
see~\cite{LiesenZur2018a}.
As it turns out, the number of (local) solutions of $f(z) = \eta$ close to 
$z_0$ can be $0, 1, 2$, or $3$, depending on $\eta$ and $z_0$.

Let $f$ be a general harmonic mapping with a singular zero $z_0$.
We prove existence of solutions of the equation $f(z) = \eta$ by showing 
convergence of the harmonic Newton iteration with suitable initial points to 
zeros of $f(z) - \eta$.
Depending on $\eta$, the shape of the basins of attraction and the number of 
solutions can change dramatically.

For simplicity we transform $f = h + \conj{g}$ in a local normal form.  Let 
$z_0 \in \cC$ be a singular zero of $f$ with $\abs{h'(z_0)} = \abs{g'(z_0)} 
\neq 0$; see~\eqref{eqn:crit}.
We then have
\begin{equation*}
f(z) = h(z) + \conj{g(z)}
= \sum_{k=0}^\infty a_k (z-z_0)^k + \conj{\sum_{k=0}^\infty b_k (z-z_0)^k}
\end{equation*}
with $0 = f(z_0) = a_0 + \conj{b_0}$, and $0 \neq \conj{b_1} = a_1 e^{i 2 
\theta}$,
$\theta \in \co{0, \pi}$.  Hence,
\begin{equation*}
f(z) = \frac{\conj{b_1}}{e^{i \theta}} \left( \sum_{k=1}^\infty 
\frac{a_k}{\conj{b_1}} 
e^{i (k+1) \theta} (e^{-i \theta} (z-z_0))^k 
+ \conj{\sum_{k=1}^\infty \frac{b_k}{b_1 } e^{i (k-1) \theta} 
(e^{-i \theta} (z-z_0))^k } \right),
\end{equation*}
and substituting
\begin{equation} \label{eqn:subst}
\zeta = e^{-i \theta} (z-z_0), \quad
\alpha_k = \frac{a_k}{\conj{b_1}} e^{i(k+1)\theta} = \frac{a_k}{a_1} e^{i 
(k-1) \theta}, \quad
\beta_k = \frac{b_k}{b_1} e^{i (k-1) \theta},
\end{equation}
gives $f(z) = \conj{b}_1 e^{-i \theta} \widetilde{f}(\zeta)$, where 
$\widetilde{f}$ is the unique \emph{local normal form} of $f$ at $z_0$,
\begin{equation} \label{eqn:normal}
\widetilde{f}(\zeta) = \zeta + \conj{\zeta} + \sum_{k=2}^\infty \alpha_k \zeta 
^k + \conj{\sum_{k = 2}^\infty \beta_k \zeta^k}.
\end{equation}
Here we used translation ($\zeta_0 = 0$), and scaling and rotation ($\alpha_1 = 
\beta_1 = 0$).

We take solutions of the truncated problem
$\zeta + \conj{\zeta} + \alpha_2 \zeta^2 + \conj{\beta}_2 \conj{\zeta}^2 
= \eta$
that are close to the origin as initial points for the harmonic Newton method 
applied to $\widetilde{f}(\zeta) - \eta$.

\begin{lemma} \label{lem:zero_at_crit}
Let $f = h + \conj{g}$ be as in~\eqref{eqn:normal},
\begin{equation*}
f(\zeta) = \zeta + \conj{\zeta} + \sum_{k=2}^\infty \alpha_k \zeta^k + 
\conj{\sum_{k=2}^\infty \beta_k \zeta^k}.
\end{equation*}
In particular, $f$ has a singular zero at $0$ with $\abs{h'(0)} = \abs{g'(0)} 
\neq 0$.
Furthermore, let $f_{\delta c}(\zeta) = f(\zeta) - \delta c$ with $c = 
-(\alpha_2 + \conj{\beta}_2)$ and real $\delta \in \R$.
\begin{enumerate}
\item For $\im(c) \neq 0$ and sufficiently small $\delta > 0$,
the function $f_{\delta c}$ has two distinct zeros close to $0$.
These zeros are the limits of the harmonic Newton iteration for $f_{\delta c}$
with the initial points $\zeta_\pm = \pm i \sqrt{\delta}$.

\item For $\im(c) = 0$, $\abs{\alpha_2} \neq \abs{\beta_2}$, and
sufficiently small $\abs{\delta}$, 
the harmonic Newton iteration for $f_{\delta c}$ with initial point $\zeta_3 = 
\frac{1 - \sqrt{1 - \delta c^2}}{c}$ converges to a zero of $f_{\delta c}$.
\end{enumerate}
\end{lemma}

\begin{proof}
To apply the Newton-Kantorovich theorem to $f_{\delta c}$ and each initial 
point, we estimate $\alpha$ and $\omega_0$; see~\eqref{eqn:alpha} 
and~\eqref{eqn:omega0}.

Let $\im(c) \neq 0$.  Since $\zeta_\pm = \pm i \sqrt{\delta}$ 
solves $\zeta + \conj{\zeta} + \alpha_2 \zeta^2 + \conj{\beta}_2 \conj{\zeta}^2 
= \delta c$, we have
\begin{equation*}
\abs{f_{\delta c}(\zeta_\pm)}
\leq (\abs{\alpha_3} + \abs{\beta_3}) \abs{\zeta_\pm}^3 + 
\cO(\abs{\zeta_\pm}^4).
\end{equation*}
From
\begin{equation*}
h'(\zeta) = 1 + 2 \alpha_2 \zeta + \cO(\zeta^2), \quad
g'(\zeta) = 1 + 2 \beta_2 \zeta + \cO(\zeta^2),
\end{equation*}
we have, using $\conj{\zeta}_\pm = - \zeta_\pm$,
\begin{equation*}
\abs{h'(\zeta_\pm)}^2 = h'(\zeta_\pm) \conj{h'(\zeta_\pm)}
= 1 + 4i \im(\alpha_2) \zeta_\pm + \cO(\abs{\zeta_\pm}^2),
\end{equation*}
and similarly for $g$.  Thus
\begin{align*}
\abs[\big]{\abs{h'(\zeta_\pm)} - \abs{g'(\zeta_\pm)}}
&= \frac{\abs[\big]{\abs{h'(\zeta_\pm)}^2 - 
\abs{g'(\zeta_\pm)}^2}}{\abs{h'(\zeta_\pm)} + \abs{g'(\zeta_\pm)}}
= \frac{4 \abs{\im(\alpha_2 - \beta_2) \zeta_\pm} + \cO(\abs{\zeta_\pm}^2)}{2 + 
\cO(\abs{\zeta_\pm})} \\
&= 2 \abs{\im(\alpha_2 - \beta_2)} \abs{\zeta_\pm} + \cO(\abs{\zeta_\pm}^2) \\
&= 2 \abs{\im(c)} \abs{\zeta_\pm} + \cO(\abs{\zeta_\pm}^2),
\end{align*}
for sufficiently small $\delta > 0$.  Together we find
\begin{align*}
\frac{\abs{f_{\delta c}(\zeta_\pm)}}{\abs[\big]{\abs{h'(\zeta_\pm)} - 
\abs{g'(\zeta_\pm)}}}
&\le \frac{(\abs{\alpha_3} + \abs{\beta_3})\abs{\zeta_\pm}^3 + 
\cO(\abs{\zeta_\pm}^4)}{2 \abs{\im(c)} \abs{\zeta_\pm} + 
\cO(\abs{\zeta_\pm}^2)} \\
&= \frac{\abs{\alpha_3} + \abs{\beta_3}}{2\abs{\im(c)}} \abs{\zeta_\pm}^2 + 
\cO(\abs{\zeta_\pm}^3)
= \alpha.
\end{align*}
Next, we estimate $\omega_0$ on $D_\pm = \{ \zeta \in \C: \abs{\zeta - 
\zeta_\pm} < q \abs{\zeta_\pm} \}$ for a $0 < q < 1$.  From
\begin{equation*}
h''(\zeta) = 2 \alpha_2 + \cO(\zeta), \quad
g''(\zeta) = 2 \beta_2 + \cO(\zeta),
\end{equation*}
we have
\begin{equation*}
\sup_{\tau \in D_\pm} \abs{h''(\tau)} + \sup_{\tau \in D_\pm} \abs{g''(\tau)}
\leq 2 (\abs{\alpha_2} + \abs{\beta_2}) + \cO(\abs{\zeta_\pm}),
\end{equation*}
and
\begin{equation*}
\begin{split}
\frac{\sup_{\tau \in D_\pm} \abs{h''(\tau)} + \sup_{\tau \in D_\pm} 
\abs{g''(\tau)}}{\abs[\big]{\abs{h'(\zeta_\pm)} - \abs{g'(\zeta_\pm)}}}
&\leq \frac{2 (\abs{\alpha_2} + \abs{\beta_2}) + 
\cO(\abs{\zeta_\pm})}{2\abs{\im(c)}
\abs{\zeta_\pm} + \cO(\abs{\zeta_\pm}^2)} \\
&=  \frac{\abs{\alpha_2} + \abs{\beta_2}}{\abs{\im(c)} \abs{\zeta_\pm}} + \cO(1)
= \omega_0.
\end{split}
\end{equation*}
Thus
\begin{equation*}
h_0 = \alpha \omega_0
= \frac{(\abs{\alpha_2} + \abs{\beta_2})(\abs{\alpha_3} + \abs{\beta_3})}{2 
\abs{\im(c)}^2} 
\abs{\zeta_\pm} + \cO(\abs{\zeta_\pm}^2),
\end{equation*}
so that $h_0 < \frac{1}{2}$ for sufficiently small $\delta$, 
since $\zeta_\pm = \pm i \sqrt{\delta} \to 0$ for $\delta \to 0$.
Finally, we show $\rho < q \abs{\zeta_\pm}$.
Using~\eqref{eqn:rho_minus_asympt} we find that
\begin{equation}
\rho = \alpha + \cO(\alpha h_0)
= \frac{\abs{\alpha_3} + \abs{\beta_3}}{2\abs{\im(c)}} \abs{\zeta_\pm}^2 +
\cO(\abs{\zeta_\pm}^3) < q \abs{\zeta_\pm},
\end{equation}
provided $\zeta_\pm$ is sufficiently small.  Then, by 
Theorem~\ref{thm:kantorovich},
the sequence of harmonic Newton iterates remains in the closed disks
$\conj{D}(\zeta_\pm;\rho)$
and converges to a zero of 
$f_{\delta c}$.  Furthermore, both disks are disjoint, which implies that the 
iterations converge to two distinct zeros of $f_{\delta c}$.

For the second part, assume that $\im(c) = 0$ and $\abs{\alpha_2} \neq 
\abs{\beta_2}$.
The latter implies that $c \neq 0$.
The only difference to the first case is in the estimate of
$\abs[\big]{\abs{h'(\zeta_3)} - \abs{g'(\zeta_3)}}$.
Since $c$ is real, also $\zeta_3 = \frac{1 - \sqrt{1 - \delta c^2}}{c} = 
\frac{c}{2} \delta + \cO(\delta^2)$ is real for $\delta \leq \frac{1}{c^2}$, 
and we find from $h'(\zeta) = 1 + 2 \alpha_2 \zeta + \cO(\zeta^2)$ that
\begin{equation*}
\abs{h'(\zeta_3)}^2
= h'(\zeta_3) \conj{h'(\zeta_3)}
= 1 + 4 \re(\alpha_2) \zeta_3 + \cO(\abs{\zeta_3}^2),
\end{equation*}
and similarly for $g$.  Then
\begin{align*}
\abs[\big]{\abs{h'(\zeta_3)} - \abs{g'(\zeta_3)}}
&= \frac{\abs[\big]{\abs{h'(\zeta_3)}^2 - 
\abs{g'(\zeta_3)}^2}}{\abs{h'(\zeta_3)} + \abs{g'(\zeta_3)}}
= \frac{\abs{4 \re(\alpha_2 - \beta_2) \zeta_3 + \cO(\zeta_3^2)}}{2 + 
\cO(\abs{\zeta_3})} \\
&= 2 \abs{\re(\alpha_2 - \beta_2)} \abs{\zeta_3} + \cO(\abs{\zeta_3}^2).
\end{align*}
The assumptions $\abs{\alpha_2} \neq \abs{\beta_2}$ and $\im(c) = 0$ guarantee 
that $\re(\alpha_2 - \beta_2) \neq 
0$ and thus that $\abs[\big]{\abs{h'(\zeta_3)} - \abs{g'(\zeta_3)}} = 
\cO(\abs{\zeta_3})$.
We then obtain $\alpha$, $\omega_0$ in the same orders of magnitude as in the 
first case, so that $h_0 = \alpha \omega_0 = \cO(\abs{\zeta_3}) < \frac{1}{2}$ 
and
$\rho < q \abs{\zeta_3}$ for sufficiently small $\abs{\delta}$.  Therefore, the 
harmonic Newton iteration with initial point $\zeta_3$ converges to a zero of 
$f_{\delta c}$.
\end{proof}

\begin{theorem} \label{thm:zero_at_crit}
Let the harmonic mapping
\begin{equation*}
f(z) = h(z) + \conj{g(z)}
= \sum_{k=1}^\infty a_k (z-z_0)^k + \conj{\sum_{k=1}^\infty b_k (z-z_0)^k}
\end{equation*}
have a singular zero at $z_0$ with $\abs{a_1} = \abs{b_1} \neq 0$.
Furthermore, let $\theta \in \co{0, \pi}$ be defined by $\conj{b}_1 = a_1 e^{i 
2 \theta}$, and let
\begin{equation*}
f_{\delta c}(z) = f(z) - \delta c, \quad \text{with} \quad 
c = \conj{b}_1 e^{-i \theta} \widetilde{c}
\quad \text{and} \quad
\widetilde{c} = - \frac{a_2}{a_1} e^{i \theta} - \conj{ \left( \frac{b_2}{b_1} 
e^{i \theta} \right) },
\end{equation*}
where $\delta$ is real.
\begin{enumerate}
\item For $\im(\widetilde{c}) \neq 0$ and sufficiently small $\delta > 0$,
the function $f_{\delta c}$ has two distinct zeros close to $z_0$.
These zeros are the limits of the 
harmonic Newton iteration for $f_{\delta c}$
with the initial points 
\begin{equation*}
z_\pm = z_0 \pm i \sqrt{\delta} e^{i \theta}
= z_0 \pm i \sqrt{\delta \frac{\conj{b}_1}{a_1}}.
\end{equation*}

\item For $\im(\widetilde{c}) = 0$, $\abs{a_2} \neq
\abs{b_2}$, and sufficiently small $\abs{\delta}$,  the harmonic Newton 
iteration with initial point
\begin{equation*}
z_3 = z_0 + \frac{1 - \sqrt{1 - \delta \widetilde{c}^2}}{\widetilde{c}} e^{i 
\theta},
\end{equation*}
converges to a zero of $f_{\delta c}$.
\end{enumerate}
\end{theorem}

\begin{proof}
We transform $f$ to its local normal form $\widetilde{f}$ at $z_0$, 
recall~\eqref{eqn:subst}, \eqref{eqn:normal}, and $f(z) = \conj{b}_1 e^{-i 
\theta} \widetilde{f}(\zeta)$.
Then $f(z) - \delta c = \conj{b}_1 e^{- i \theta} ( \widetilde{f}(\zeta) - 
\delta \widetilde{c})$ with
\begin{equation*}
\widetilde{c} = - \frac{a_2}{a_1} e^{i \theta} - \conj{\left( \frac{b_2}{b_1} 
e^{i \theta} \right)}
= - (\alpha_2 + \conj{\beta}_2).
\end{equation*}

Let $\im(\widetilde{c}) \neq 0$.  Then, by 
Lemma~\ref{lem:zero_at_crit}, the harmonic Newton 
iteration with initial points $\zeta_\pm = \pm i \sqrt{\delta}$ converges to 
zeros of $\widetilde{f}_{\delta \widetilde{c}}$, when $\delta > 0$ is 
sufficiently small.  Back transformation gives the desired result; 
see~\eqref{eqn:subst}.

Let $\im(\widetilde{c}) = 0$, and $\abs{\alpha_2}
\neq \abs{\beta_2}$, which is equivalent to $\abs{a_2} \neq \abs{b_2}$.  Then, 
by Lemma~\ref{lem:zero_at_crit}, the harmonic 
Newton iteration with initial point $\zeta_3 = \frac{1 - \sqrt{1 - \delta 
\widetilde{c}^2}}{\widetilde{c}}$ converges to a zero of $\widetilde{f}_{\delta
\widetilde{c}}$, when $\delta > 0$ is sufficiently small.  Back transformation 
gives the desired result.
\end{proof}

Varying $\delta$ can lead to a bifurcation of the solutions of $f(z) = \delta 
c$, and the basins of attraction can merge or split; see 
Figure~\ref{fig:bifurcation}.
\begin{figure}[t!]
{\centering
\includegraphics[width=0.495\linewidth]{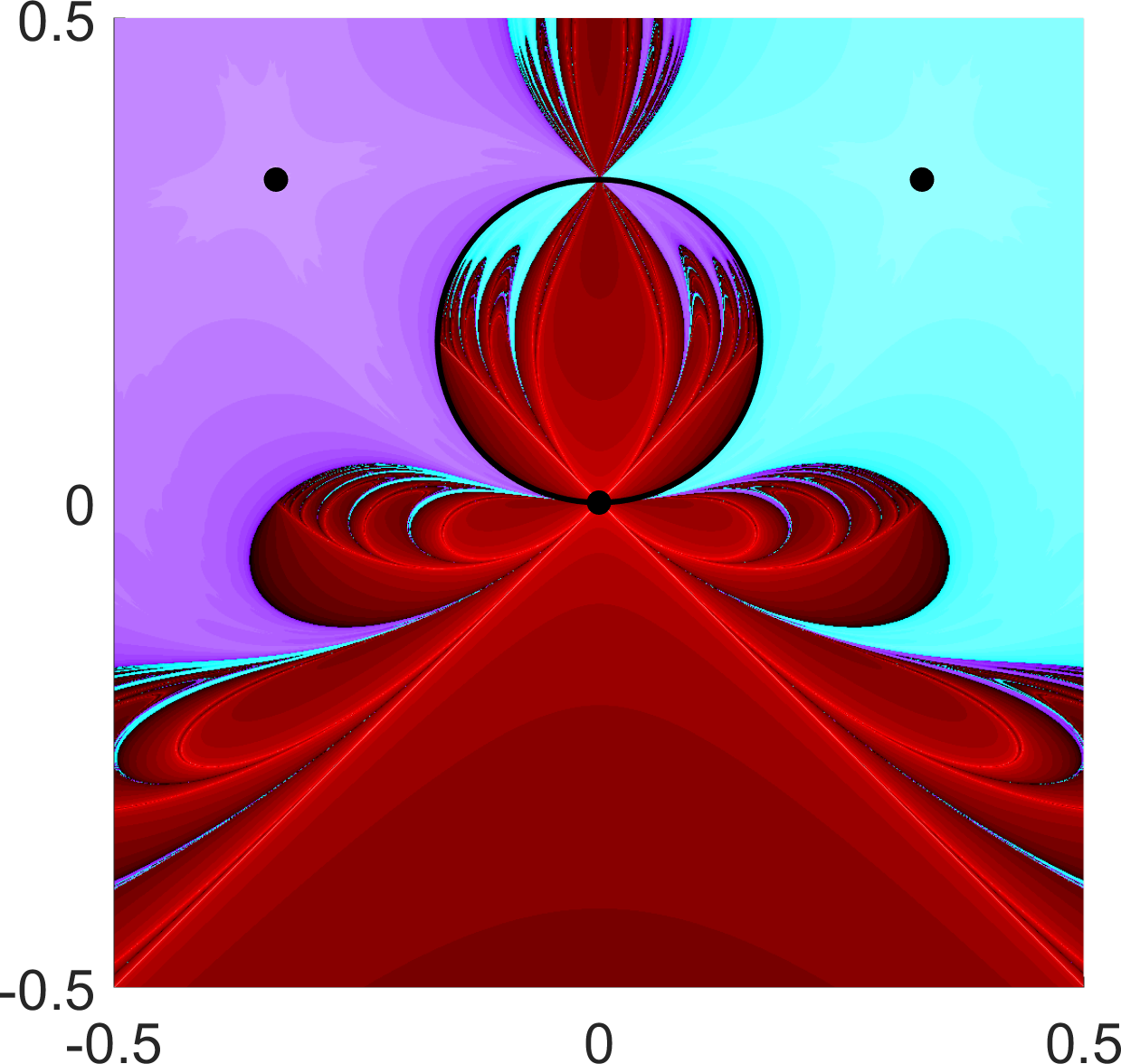}
\includegraphics[width=0.495\linewidth]{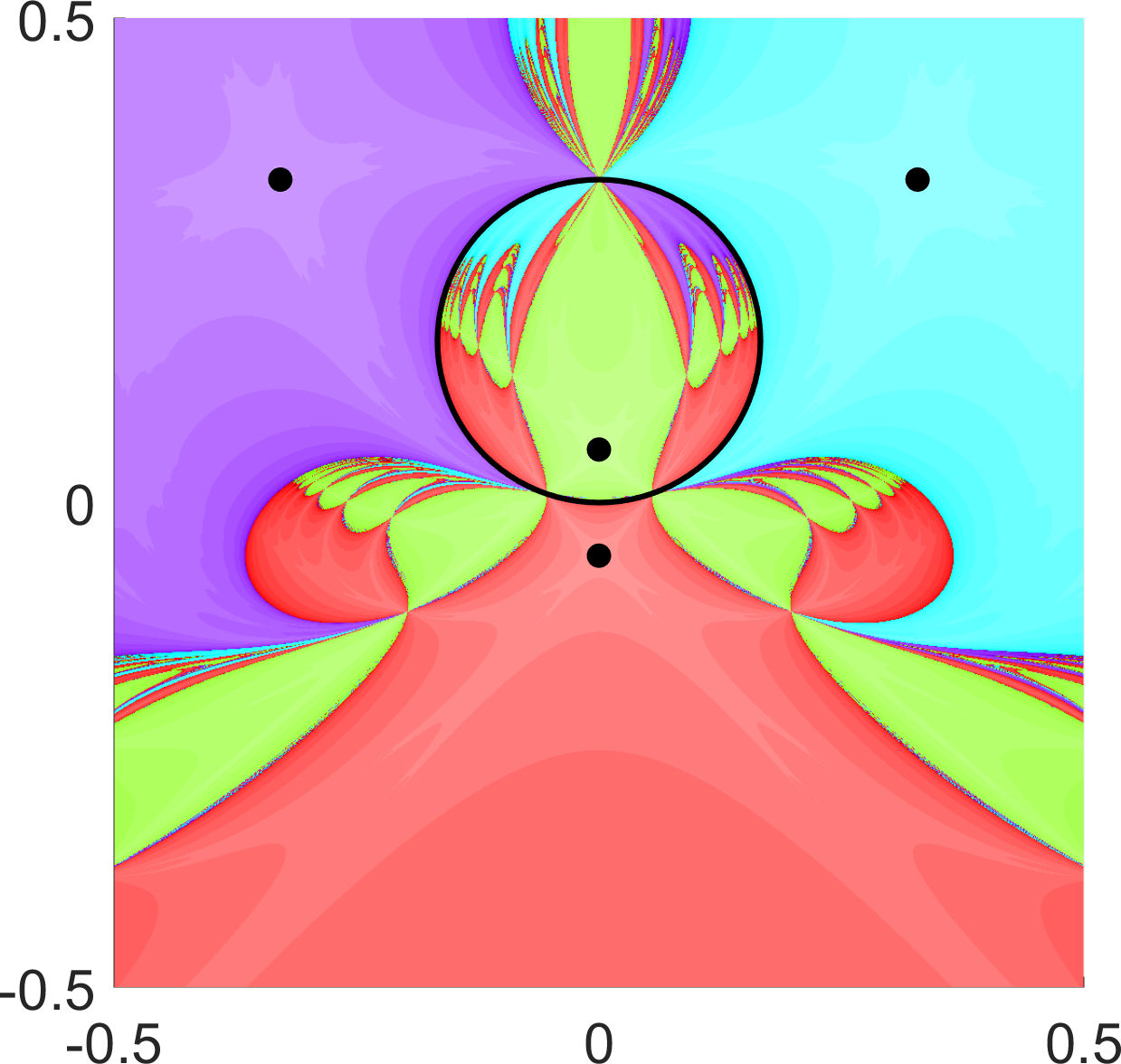}

\includegraphics[width=0.495\linewidth]{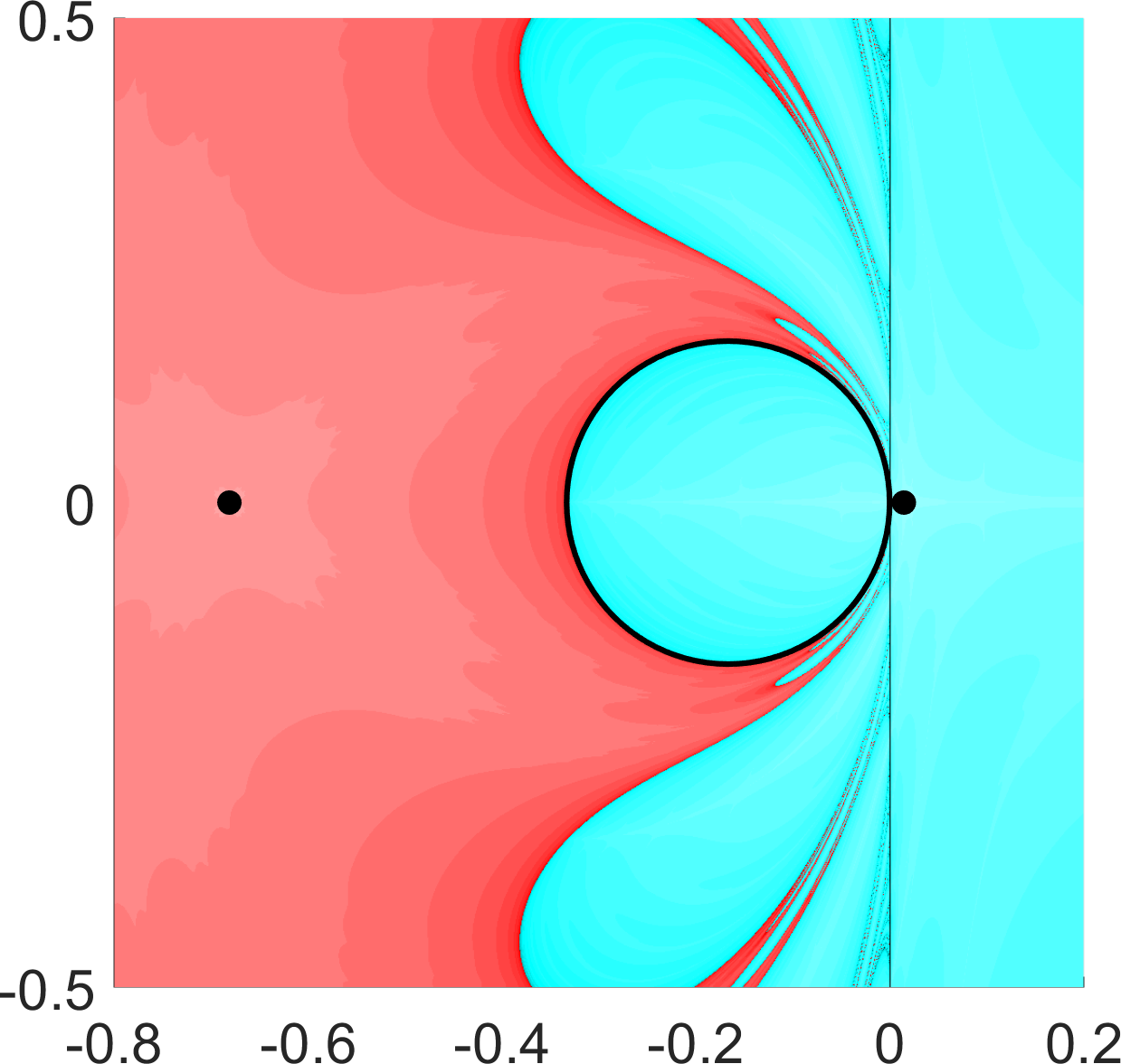}
\includegraphics[width=0.495\linewidth]{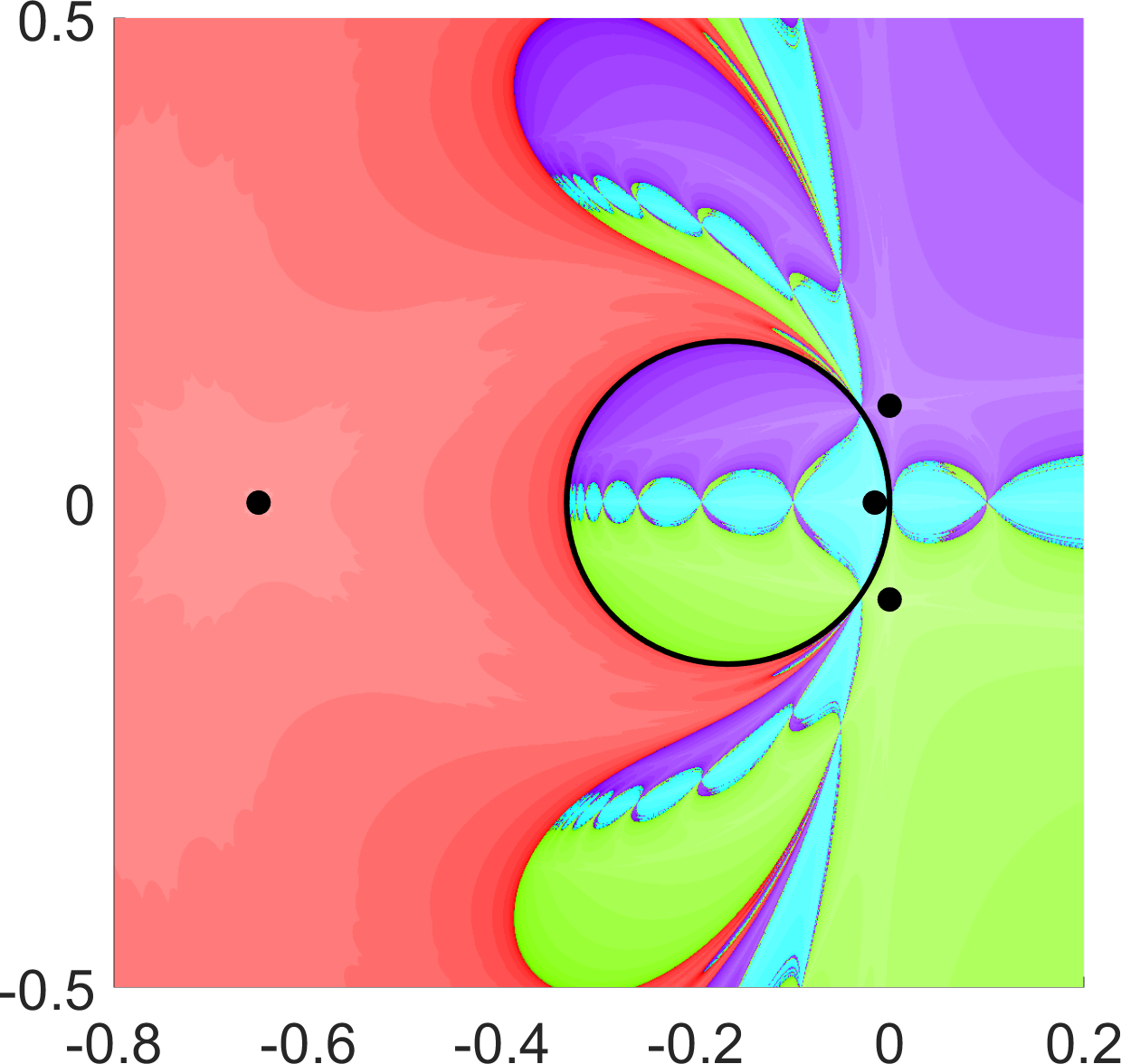}
\vspace{2mm}

}
\caption{Basins of attraction.
Top: $f_{- i\delta}(z) = z + \conj{z} + 2i z^2 + \conj{i z^2} + i \delta$ 
with $\delta = 0$ (left) and $\delta = 0.003$ (right).
Bottom: $f_{-3 \delta}(z) = z + \conj{z} + 2 z^2 + \conj{z}^2 + 3 \delta$ 
with $\delta = - 0.01$ (left) and $\delta = 0.01$ (right).
Black lines mark the critical sets.}
\label{fig:bifurcation}
\end{figure}

\begin{remark}
\begin{enumerate}
\item The conditions $\im(c) \neq 0$ in Lemma~\ref{lem:zero_at_crit} and 
$\im(\widetilde{c}) \neq 0$ in Theorem~\ref{thm:zero_at_crit}
imply that $f(z_0)$ is not a 
\emph{cusp} of the \emph{caustics} of $f$; see~\cite{LiesenZur2018a}.

\item In Figure~\ref{fig:bifurcation}, the singular zero at $z=0$ (top left) 
splits into two non-singular zeros (top right).
Note that the harmonic Newton iteration converges quicker to 
the non-singular zeros than to the singular zero.

\item For $f$ in Figure~\ref{fig:bifurcation} (bottom left) we have 
$H_f(iy) = \frac{i}{2}(y - \frac{0.01}{y})$, $y \in \R$.
Hence, the harmonic Newton iteration diverges for initial points on the 
imaginary axis, which explains the black line.

\item Figure~\ref{fig:bifurcation} (bottom right) suggests that we have zeros 
of 
$f_{\delta c}$ close to $z_\pm = \pm i \sqrt{\delta}$ also in the case $\im(c) 
= 0$ and $\delta > 0$.  However, without further improvements, the above proof 
strategy breaks down for $z_\pm$ as initial points.
\end{enumerate}
\end{remark}

\section{Further Examples}\label{sect:examples}

We illustrate the harmonic Newton method with further examples, using the 
MATLAB implementation in Figure~\ref{fig:code}.

\subsection{Harmonic polynomials}
\label{sect:harmonic_polynomials}

Harmonic polynomials are of the form $f = p + \conj{q}$, where $p$ and $q$ are 
(analytic) polynomials with respective degrees $n$ and $m$.
Wilmshurst conjectured in~\cite{Wilmshurst1998} that the number of zeros of $f$ 
for $n > m$ fulfills
\begin{equation}\label{eqn:wilmshurst_conjecture}
N(f) \leq 3n-2 + m(m-1),
\end{equation}
and proved it for  $m = n-1$, including sharpness of the bound.
For $n = m$, Wilmshurst showed the bound $N(f) \leq n^2$, provided that 
$\lim_{z \to \infty} f(z) = \infty$ holds.  Otherwise $f$ could have infinitely 
many zeros; compare Lemma~\ref{lem:deg_one}. 
Later, the case $m = 1$ was settled in~\cite{KhavinsonSwiatek2003}, and 
sharpness of the bound in~\cite{Geyer2008}.
For several other values of $1 < m < n-1$, the conjecture was shown to be
wrong; see \cite{LeeLerarioLundberg2015,HauensteinLerarioLundbergMehta2015,KhavinsonLeeSaez2018}.

\begin{example} \label{ex:wilmshurst}
We consider the original example of Wilmshurst in~\cite{Wilmshurst1998} showing 
sharpness of~\eqref{eqn:wilmshurst_conjecture} for $m = n-1$. Let 
\begin{equation} \label{eqn:wilmshurst}
f(z) = p(z) + \conj{q(z)}
= z^n + (z-1)^n + \conj{i (z-1)^n -i z^n},
\end{equation}
which has $\deg(p) = n$, and $\deg(q) = n-1$, and $n^2$ zeros.
\begin{figure}[t!]
{\centering
\includegraphics[width=0.48\linewidth]{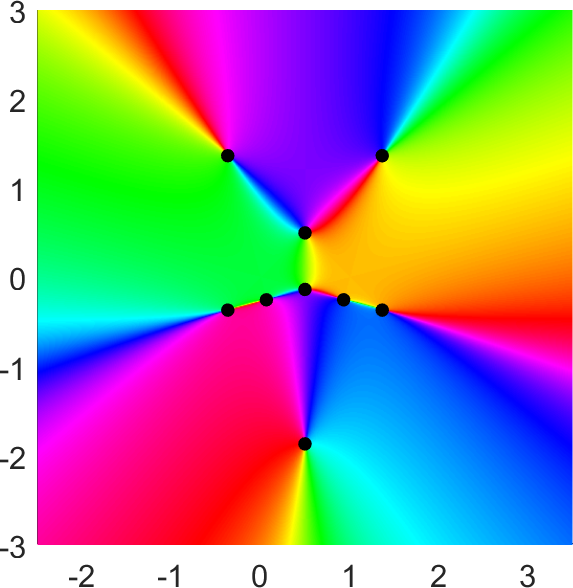}
\includegraphics[width=0.48\linewidth]{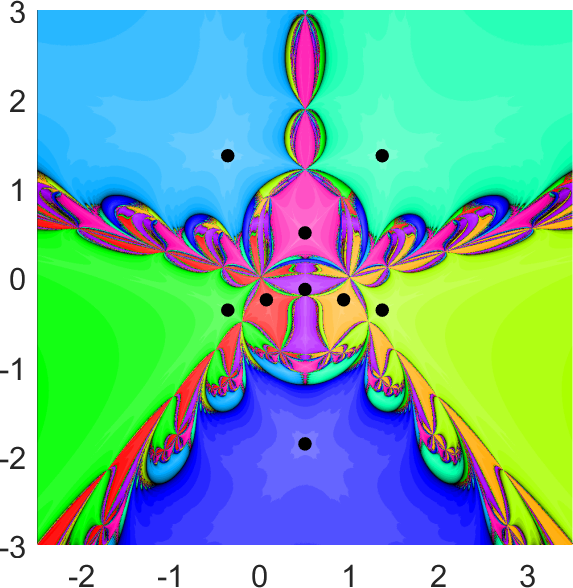}
\vspace{2mm}

\includegraphics[width=0.48\linewidth]{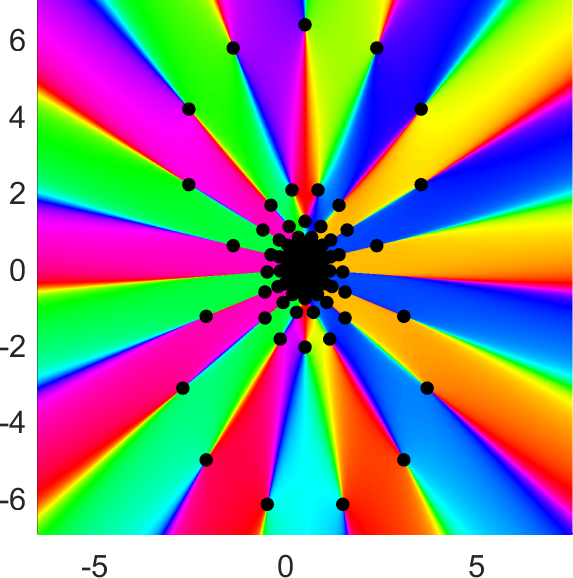}
\includegraphics[width=0.48\linewidth]{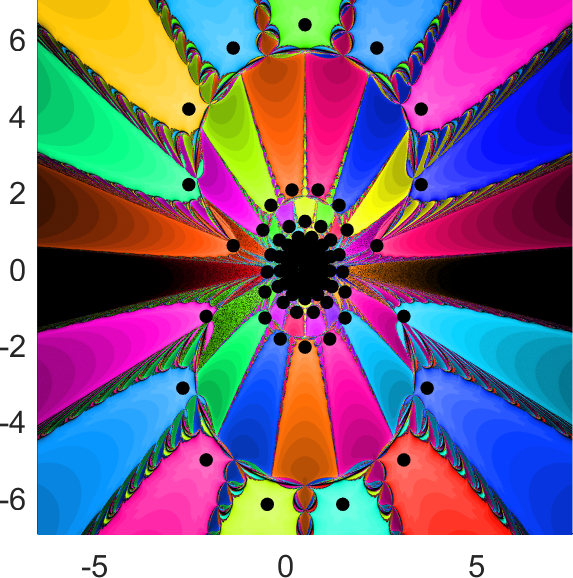}
\vspace{2mm}

}
\caption{Phase plots (left) and basins of attraction (right) for
Wilmshurst's harmonic polynomial~\eqref{eqn:wilmshurst} for $n = 3$ 
(top) and $n = 10$ (bottom).}
\label{fig:wilmshurst}
\end{figure}

To compute the zeros of $f$ with the harmonic Newton method, we use a grid of 
initial 
points with mesh size $0.05$ in the plot region in Figure~\ref{fig:wilmshurst}.
On the left we have the phase plots, and on the right the basins of attraction
for $n = 3$ and the $n^2 = 9$ zeros (top), as well as for $n = 10$ with
its $100$ zeros (bottom).
The plots are centered at $z = 0.5$, since $z \mapsto p(z+0.5)$ and $z 
\mapsto q(z+0.5)$ are even or odd, depending on $n$.

For $n = 3$, the maximal residual of $f$ at the computed zeros is $9.8625 \cdot 
10^{-15}$.
For $n = 10$, we have $\abs{f(z_j)} \leq 3.0146 \cdot 10^{-14}$ at computed 
zeros with $\abs{z_j - 0.5} \leq 1$, and $\abs{f(z_j)} \leq 1.3738 \cdot 
10^{-7}$ at zeros with $\abs{z_j - 0.5} > 1$.
The latter rather large absolute residual is an effect of floating point 
arithmetic.  Indeed, the magnitude of $z^n$ and $(z-1)^n$ is of order $10^8$ at 
the outer zeros, which results in a loss of accuracy of about $8$ digits when 
evaluating $f$.  Evaluating the polynomials $p$ and $q$ with Horner's 
scheme~\cite[Sect.~5.1]{Higham2002} does not alleviate this problem, since we 
still need to compute $p + \conj{q}$.  However, the maximum of the relative 
residuals 
$\abs{f(z_j)}/\abs{z_j-0.5}^{10}$ is $1.4010 \cdot 10^{-14}$, which is quite 
satisfactory.
\end{example}

\subsection{Gravitational lensing}

Gravitational lensing can 
be modeled with harmonic mappings, where positions of lensed images are zeros 
of these functions;
see~\cite{KhavinsonNeumann2008,Petters2010,BeneteauHudson2018}.
We consider point mass lenses and isothermal gravitational lenses.

Example~\ref{ex:mpw} already featured a rational harmonic function from 
gravitational point mass models.
Rhie~\cite{Rhie2003} constructed from this example a gravitational 
lens with the maximum number of lensed images, or equivalently a rational 
harmonic function with the maximum number of zeros, showing sharpness of the 
bound in~\cite{KhavinsonNeumann2006}.

\begin{example} \label{ex:rhie}
Rhie's function
\begin{equation}
f(z)
= (1-\eps) \frac{z^{n-1}}{z^n - r^n} + \frac{\eps}{z} 
- \conj{z} \label{eqn:rhie}
\end{equation}
is obtained from~\eqref{eqn:mpw} by adding a pole at the origin (and 
normalizing the mass to be one).
\begin{figure}[t!]
{\centering
\includegraphics[width=0.48\linewidth]{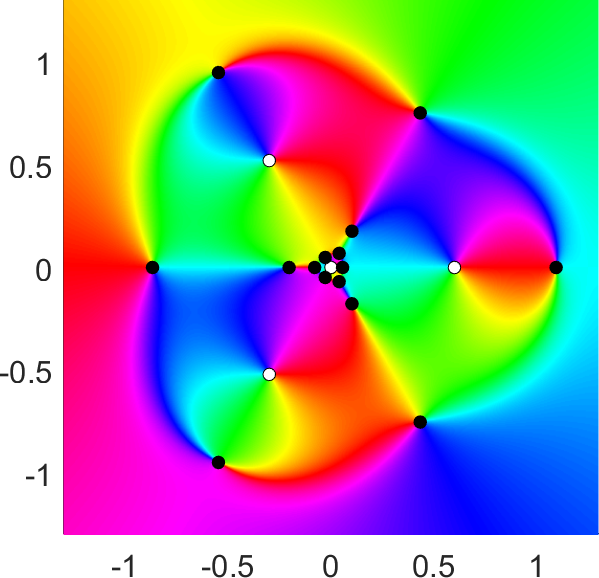}
\includegraphics[width=0.48\linewidth]{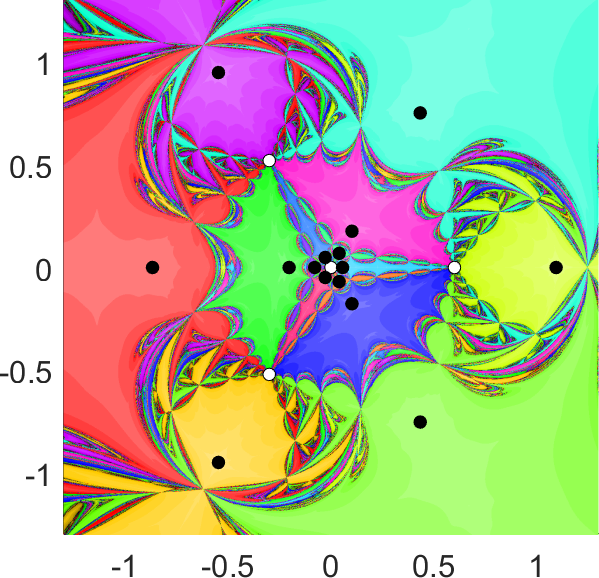}
\vspace{2mm}

}
\caption{Phase plot (left) and basins of attraction (right) for Rhie's
function~\eqref{eqn:rhie} with $r = 0.6$, $n = 3$, 
and $\eps = 0.004$; compare Figure~\ref{fig:mpw}.}
\label{fig:rhie}
\end{figure}
For sufficiently small $\eps$, it has $5n$ zeros; for a rigorous 
proof and quantification of the parameters $\eps$ and $r$
see~\cite{LuceSeteLiesen2014a}.  To compute the zeros of $f$ with $n = 3$, $r = 
0.6$ and $\eps = 0.004$, we apply the harmonic Newton method to a grid of 
initial points (mesh size $0.05$); see Figure~\ref{fig:rhie}.
The maximal residual at the computed zeros is $9.9371 \cdot 10^{-15}$.
\end{example}

\begin{example} \label{ex:einstein}
A point mass lens with a single mass is known as the Chang-Refsdal lens and 
produces an Einstein ring; see~\cite{AnEvans2006}.
It can be modeled by $f(z) = \frac{1}{z} - \conj{z}$.
The zero set of $f$ is the whole unit circle, 
in particular the zeros are not isolated.  Moreover, each zero is singular.
\begin{figure}[t!]
{\centering
\begin{minipage}{.66\linewidth}
\includegraphics[width=0.48\linewidth]{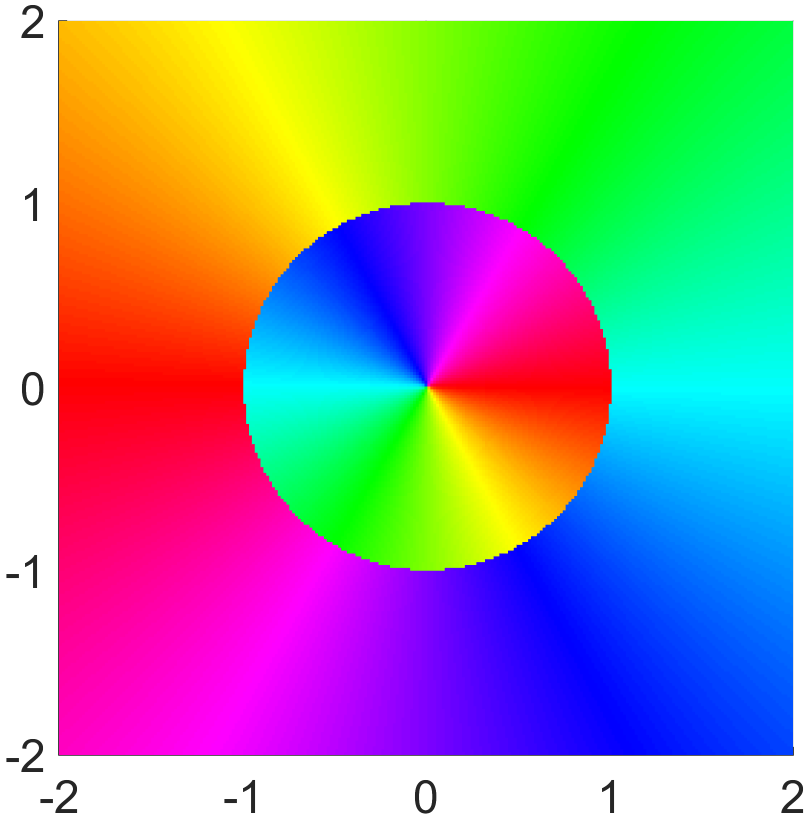}
\includegraphics[width=0.48\linewidth]{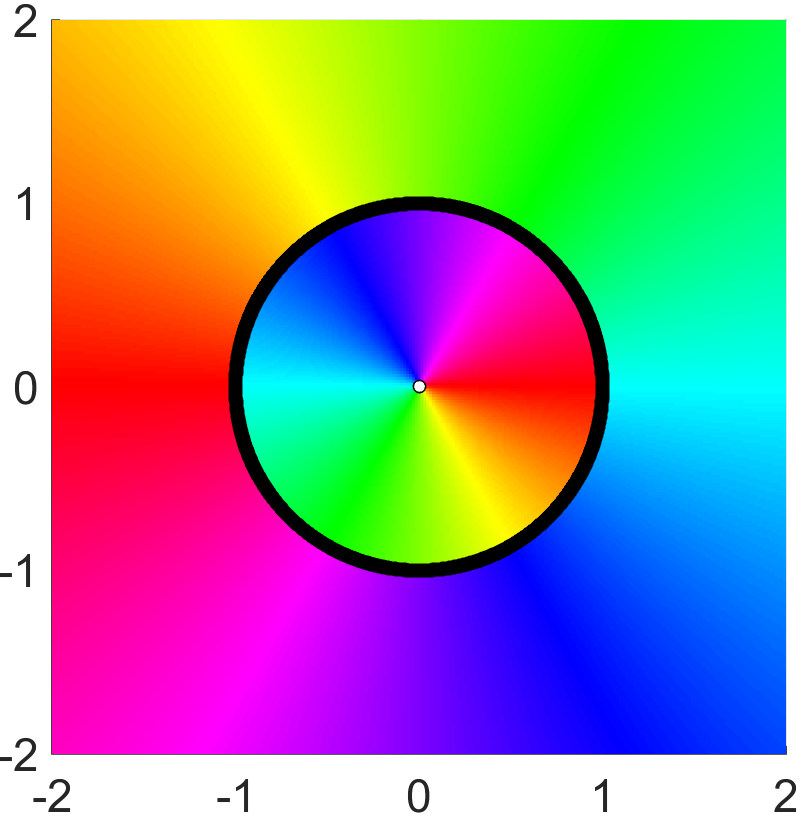}
\end{minipage}
\begin{minipage}{.3\linewidth}
\vspace{-.25cm}

\includegraphics[width=0.978\linewidth]{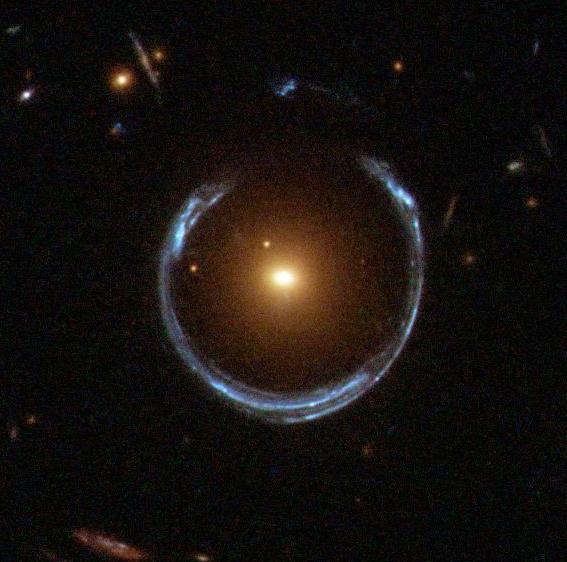}
\end{minipage}
\vspace{2mm}
}

\caption{Phase plots of $f(z) = 1/z - \conj{z}$.  The middle panel shows 
the computed zeros of $f$; see Example~\ref{ex:einstein}. The right panel shows 
an actual gravitional lens (image credit: ESA/Hubble \& NASA).}
\label{fig:einstein}
\end{figure}
To compute zeros of $f$, we take a grid of initial points in $\cc{-2, 2}^2$ 
(mesh size $0.02$).
Figure~\ref{fig:einstein} shows the iterates after at most $30$ steps 
(the mean number of steps is $6.4$). The residuals satisfy $\abs{f(z_k)} \leq 
2.4887 \cdot 10^{-14}$.
The harmonic Newton iteration for $f$ simplifies to $z_{k+1} = 
\frac{2}{1 + \abs{z_k}^2} z_k$ for $\abs{z_k} \neq 1$, showing that $z_k \to 
e^{i \varphi}$ if
$z_0 = r_0 e^{i \varphi} \neq 0$, i.e., the basin of attraction of $e^{i 
\varphi}$ is
\begin{equation*}
A(e^{i \varphi}) = \{ z = r e^{i \varphi} : r > 0, r \neq 1 \}.
\end{equation*}
Note that $A(e^{i \varphi})$ does not contain an open subset around $e^{i 
\varphi}$, which is due to the fact that the zeros of $f$ are not isolated.
\end{example}

\begin{example} \label{ex:isothermal}
We consider gravitational lensing by an isothermal elliptical 
galaxy with compactly supported mass; 
see~\cite{BeneteauHudson2018,KhavinsonLundberg2010,BergweilerEremenko2010}.
After a change of variables, the lensed images are the zeros of the 
transcendental harmonic mapping
\begin{equation} \label{eqn:isothermal}
f(z)
= z - \arcsin \left( \frac{k}{\conj{z} + \conj{w}} 
\right),
\end{equation}
where $w$ is the position of the light source (projected onto the lens plane), 
and $k$ is a real constant related to the shape of the galaxy;
see~\cite{KhavinsonLundberg2010}.  We take the principal branch of arcsine
\begin{figure}
{\centering
\includegraphics[width=0.48\linewidth]{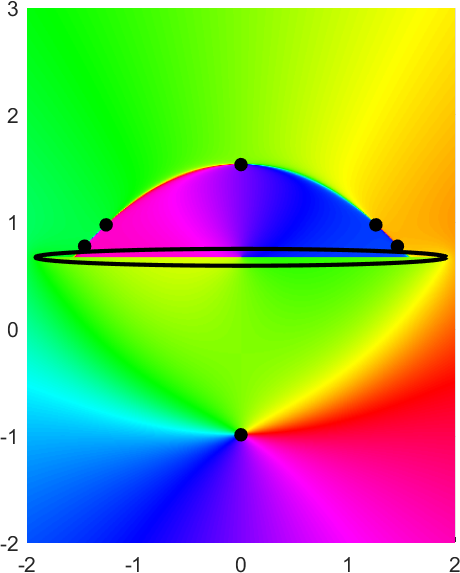}
\includegraphics[width=0.48\linewidth]{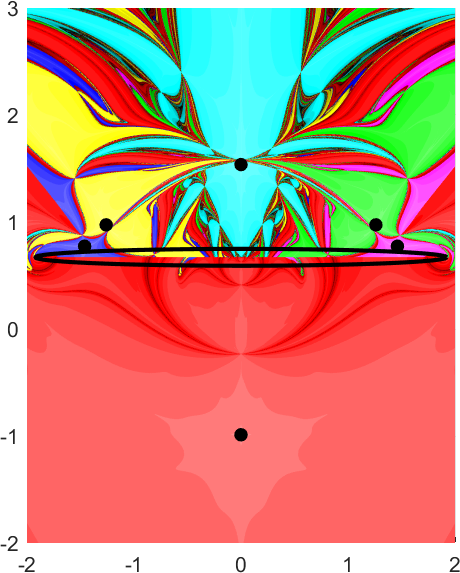}
\vspace{2mm}

}
\caption{Phase plot (left) and basins of attraction (right) for the function
$f(z) = z - \arcsin(1.92/\conj{z - 0.67i})$ in Example~\ref{ex:isothermal};
compare to~\cite[Fig.~3]{KhavinsonLundberg2010}.}
\label{fig:isothermal}
\end{figure}
and compute the zeros of $f$ by applying the harmonic Newton method to a 
grid of initial points in $\cc{-2, 2}^2$ (mesh size $0.02$).
Figure~\ref{fig:isothermal} shows the phase plot of $f$ (left) and the basins 
of attraction (right) for $k = 1.92$ and $w = -0.67i$ 
from~\cite{BergweilerEremenko2010}.
The black ellipse is the shape of the galaxy, and the sharp edge in the phase 
plot (inside the ellipse) is the branch cut of arcsine.
\end{example}

\section{Summary and Outlook}\label{sect:outlook}

We derived a complex formulation of Newton's method for harmonic mappings and,
more generally, for real differentiable functions in $\C$.
The iterations~\eqref{eqn:complex_newton} and its 
harmonic pendant~\eqref{eqn:harmonic_newton} allow to compute the zeros of 
harmonic and even non-analytic complex functions without losing any advantages 
of Newton's method.
The complex formulation makes the iteration amenable to analysis in a complex 
variables spirit.
In particular, close to a pole of a harmonic mapping $f$, we derived initial 
points for which the harmonic Newton iteration 
is guaranteed to converge to zeros of $f$.
Similarly, we derived initial points to obtain solutions of $f(z) = \eta$ for 
certain small $\abs{\eta}$ and $z$ close to a singular zero of $f$.
In both cases, the convergence proof relies on the Newton-Kantorovich theorem.
In particular, our theorems also show existence of zeros.

The harmonic Newton method should prove useful in the further study of harmonic 
polynomials, in particular related to Wilmshurst's conjecture, and more 
generally of harmonic mappings.
Visualizing the basins of attraction can give new insights into the behavior of 
the functions.

For certain classes of harmonic mappings, e.g., harmonic polynomials, it could 
be interesting to find a set $\mathcal{S}$, such that for any zero $z_*$ of 
$f$, the harmonic Newton iteration converges to $z_*$ for at least  one initial 
point in $\mathcal{S}$.  This approach is related to 
Theorem~\ref{thm:zero_at_pole} and Theorem~\ref{thm:zero_at_crit}.  In the
seminal paper~\cite{HubbardSchleicherSutherland2001}, such a set
$\mathcal{S}$ is constructed for complex (analytic) polynomials.

A further analysis of the harmonic Newton map $H_f$ in the spirit of complex 
dynamics should be of great interest, e.g., points of indeterminacy, and the 
continuation of $H_f$ (and of $f$ itself) to the Riemann sphere.
Both have not been addressed in this paper.

For analytic functions, many other iterative root finding methods exist;
see e.g.~\cite{Gilbert2001,Varona2002}.  It could be interesting to generalize 
them to harmonic mappings.
As a drawback we may have to introduce higher order derivatives for harmonic 
mappings.
The efficiency index of several of these methods for analytic functions 
are compared in~\cite{Varona2002}.  In this light, Newton's method should be a 
decent choice.

\paragraph*{Acknowledgments.}
We thank J\"org Liesen for helpful comments on the manuscript.
Moreover, we thank the anonymous referees for several helpful suggestions 
which lead to improvements of the presentation.

\footnotesize
\bibliography{harmonic_newton}
\bibliographystyle{siam} 

\end{document}